\documentclass[a4paper, twoside, 11pt]{article}
\usepackage{titlesec, longtable}

\titleclass{\subsubsubsection}{straight}[\subsection]
\newcounter{subsubsubsection}[subsubsection]

\renewcommand\thesubsubsubsection{\thesubsubsection.\arabic{subsubsubsection}}

\titleformat{\subsubsubsection}
 {\normalfont\normalsize\bfseries}{\thesubsubsubsection}{1em}{}
\titlespacing*{\subsubsubsection}
{0pt}{3.25ex plus 1ex minus .2ex}{1.5ex plus .2ex}

\makeatletter
\renewcommand\paragraph{\@startsection{paragraph}{5}{\z@}%
 {3.25ex \@plus1ex \@minus.2ex}%
 {-1em}%
 {\normalfont\normalsize\bfseries}}
\renewcommand\subparagraph{\@startsection{subparagraph}{6}{\z@}
 {3.25ex \@plus1ex \@minus .2ex}%
 {-1em}%
 {\normalfont\normalsize\bfseries}}
\def\toclevel@subsubsubsection{4}
\def\toclevel@paragraph{5}
\def\toclevel@paragraph{6}
\def\l@subsubsubsection{\@dottedtocline{4}{7em}{4em}}
\def\l@paragraph{\@dottedtocline{5}{10em}{5em}}
\def\l@subparagraph{\@dottedtocline{6}{14em}{6em}}
\@addtoreset{subsubsubsection}{section}
\@addtoreset{subsubsubsection}{subsection}
\@addtoreset{paragraph}{subsubsubsection}
\makeatother

\usepackage{a4wide, fancyhdr, amsmath, amssymb, mathtools, yfonts, bbm}
\usepackage[hidelinks]{hyperref}
\usepackage{mathrsfs}
\usepackage{graphicx}
\usepackage{tikz}
\usepackage[all]{xy}
\usepackage[utf8]{inputenc}
\usepackage{amsthm}
\usepackage[english]{babel}
\usepackage{chngcntr}
\usepackage{ifthen}
\usepackage{calc}

\usepackage{authblk}

\newcommand{\Z}{\mathbb{Z}}
\newcommand{\Q}{\mathbb{Q}}

\newcommand\GL{\mathrm{GL}(2,\mathbb{Z})}

\numberwithin{equation}{section}
\newtheorem{lemma}{Lemma}[section]
\newtheorem{theorem}[lemma]{Theorem}
\newtheorem{prop}[lemma]{Proposition}
\newtheorem{step}[lemma]{Step}

\newtheorem{corollary}[lemma]{Corollary}
\newtheorem{mydef}[lemma]{Definition}

\newtheorem{remark}[lemma]{Remark}
\newtheorem{example}[lemma]{Example}

\newtheorem{definition}[lemma]{Definition}

\title{\vspace{-\baselineskip}\sffamily\bfseries Binary forms with the same value set III. The case of ${\bf D}_3$ and ${\bf D_6}$}
\author[1]{\'Etienne Fouvry\thanks{CNRS, Laboratoire de math\' ematiques d'Orsay, Universit\' e Paris--Saclay, 91405 Orsay, France, etienne.fouvry@universite-paris-saclay.fr}}
\author[2]{Peter Koymans\thanks{Institute for Theoretical Studies, ETH Zurich, 8092 Zurich, Switzerland, peter.koymans@eth-its.ethz.ch}}
\affil[1]{Universit\'e Paris--Saclay}
\affil[2]{ETH Zurich}
\date{\today}

\begin{document}
\maketitle
 
\begin{abstract} 
Let $F, G \in \Z[X, Y]$ be binary forms of degree $\geq 3$ with automorphism groups isomorphic to the dihedral group of cardinality $6$ or $12$. We characterize exactly when $F$ and $G$ have the same value set, that is when $F$ and $G$ satisfy $F(\Z^2) = G(\Z^2)$. 
\end{abstract}

\section{Introduction}
We continue the study initiated in \cite{FK1} and \cite{FKD4} to characterize pairs $(F, G)$ of binary forms in $\Z[X, Y]$ of degree $d\geq 3$, with integral coefficients, such that $F(\Z^2) = G(\Z^2)$, but which are not $\GL$--equivalent. To be more precise, let $d \geq 3$ be a fixed integer and let ${\rm Bin}(d, \Q)$ be the set of binary forms with rational coefficients and non-zero discriminant. If the matrix 
$$
\gamma :=
\begin{pmatrix}
 a &b\\ c& d
\end{pmatrix}
$$
belongs to ${\rm GL}(2, \Q)$ and if $F=F(X, Y)$ is an element of ${\rm Bin}(d, \Q)$, we denote by $F \circ \gamma$ the form $(F \circ \gamma)(X, Y) := F(aX+bY, cX+dY)$. We say that the forms $F$ and $G$ are {\it $\GL$--equivalent} when there exists $\gamma \in \GL$ such that $F \circ \gamma = G$, and we use the notation $F \sim_{\GL} G$ in that case. Of course, we have the implication
$$
F\sim_{\GL} G \Longrightarrow F \sim_{\rm val} G,
$$
where, by definition, $F \sim_{\rm val} G$ if and only if $F(\Z^2) = G(\Z^2)$. Here we defined $F(\Z^2) = \{F(x, y) : (x, y) \in \Z^2\}$. 

We are interested in the existence and in the characterization of pairs of {\it extraordinary forms} $(F, G)$, which by definition are two forms $F$ and $G$ satisfying the two properties
$$
F(\Z^2) = G(\Z^2) \text{ and } F \not \sim_{\GL} G.
$$
In this case we say that $F$ and $G$ are \emph{linked}. The characterization of an extraordinary form $F$ deeply depends on the group of automorphisms of $F$ defined by 
$$
{\rm Aut}(F, \Q) := \{ \gamma \in {\rm GL}(2, \Q) : F \circ \gamma = F\}.
$$
Up to ${\rm GL}(2, \Q)$--conjugation, there are ten possibilities for ${\rm Aut}(F, \Q)$. With the notations of \cite[Table 1]{SX}, these ten subgroups of ${\rm GL}(2, \Q)$ are
$$
{\bf C}_1, \, {\bf C}_2,\, {\bf C}_3,\, {\bf C}_4,\, {\bf C}_6,\, {\bf D}_1, \, {\bf D}_2,\, {\bf D}_3,\, {\bf D}_4,\, {\bf D}_6.
$$
The case of the first seven subgroups was treated in \cite{FK1} and the case of ${\bf D}_4$ was solved in \cite{FKD4}. In this paper, we finish the study of value sets by investigating the case of ${\bf D}_3$ and ${\bf D}_6$. These subgroups have cardinalities $6$ and $12$ and are generated by two matrices. Setting
$$
S := \begin{pmatrix} 0 & 1 \\ 1 & 0 \end{pmatrix} \text{ and } R := \begin{pmatrix} 0 & 1 \\ -1 & -1 \end{pmatrix}
$$
we have ${\bf D}_3 = \langle S, R \rangle$ and ${\bf D}_6 = \langle S, -R \rangle$. Explicitly, we have 
\begin{equation}
\label{D3=...}
{\bf D}_3 = \{{\rm id}, \, R, \, R^2, \, S, \, SR, \, SR^2\},
\end{equation}
and 
$$
{\bf D}_6 = {\bf D}_3 \cup (-{\bf D}_3).
$$
For instance, the cubic form 
\begin{equation}
\label{XY(X+Y)}
F_0 (X, Y) := XY(X+Y)
\end{equation}
satisfies the equality ${\rm Aut}(F_0, \Q) = {\bf D}_3$. The sextic form
\begin{equation}
\label{sextic} 
F_{a, c} := aX^6 - 3aX^5Y + cX^4Y^2 + (5a - 2c)X^3Y^3 + cX^2Y^4 - 3aXY^5 + aY^6,
\end{equation}
where $a$ and $c$ are integers such that ${\rm disc} \, F_{a,c} \neq 0$, is stable by ${\bf D}_6$ (see \cite[p. 818]{S} for more details). To follow the notation introduced in \cite[\S 1.1]{FK1}, we write $(C3)$ for the condition
$$
\leqno{(C3)} :\ \ {\rm Aut}(F, \Q)\text{ is } {\rm GL}(2, \Q)\text{--conjugate with } {\bf D}_3 \text{ or } {\bf D}_6.
$$ 
The following statement gives an efficient way to detect extraordinary forms. 

\begin{corollary}
\label{illustration1bis} 
Let $d \geq 3$. Suppose that $F \in {\rm Bin}(d, \Q)$ satisfies the condition $(C3)$. Then $F$ is extraordinary if and only if ${\rm Aut}(F, \Q)$ contains an element 
$$
\sigma =\begin{pmatrix}
a &b \\ c& d
\end{pmatrix}
$$
with the following properties
\begin{enumerate}
\item the order of $\sigma$ is equal to $3$,
\item the quadruple $(a, b, c, d)$ of rational numbers satisfies one of the four properties
\begin{enumerate}
\item $(a, b, c, d) \in \Z^4$,
\item $(a, d) \in \Z^2$ and $(b, c) \in \Z \times \left(\Z + \frac{1}{2}\right)$,
\item $(a, d) \in \Z^2$ and $(b, c) \in \left(\Z + \frac{1}{2}\right) \times \Z$,
\item $(a, b, c, d) \in \left(\Z + \frac{1}{2}\right)^4$.
\end{enumerate}
\end{enumerate}
\end{corollary}

As announced in \cite[Theorem B]{FK1}, this corollary is the analogue of \cite[Corollary 1.2]{FK1}, but its proof below will contain more difficulties to solve. This is due to the complexity of the groups of automorphisms of $F$ which now are ${\bf D}_3$ and ${\bf D}_6$. This new situation leads to more intricate combinatorial problems of covering $\Z^2$ by lattices (see \S \ref{catalog}, particularly Lemmas \ref{length5} and \ref{length6}). We solve these challenges by developing new algorithms to enumerate coverings of $\Z^2$, see Appendix \ref{a1} and \ref{a2}.

\begin{example}
From Corollary \ref{illustration1bis}, we deduce that the forms $F_0$ and $F_{a, c}$ defined in \eqref{XY(X+Y)} and \eqref{sextic} are extraordinary. Now consider the form
$$
F(X, Y) := 9 \cdot \left(F_0 \circ T \right)(X, Y) = XY(X+3Y)
$$
with $T:= \begin{pmatrix} 1/3 & 0 \\ 0 & 1 \end{pmatrix}$. By the conjugation formula, we have the equality
\begin{equation}
\label{conjugation}
{\rm Aut}(F, \Q) = T^{-1}{\rm Aut}(F_0, \Q) T.
\end{equation}
By the description \eqref{D3=...} the only elements $\sigma\in {\rm Aut}(F, \Q)$ with order $3$ are $\sigma = T^{-1} RT$ and $\sigma = T^{-1} R^2 T$. The corresponding matrices have one entry with its denominator equal to $3$. So $\sigma$ has not the shape required by Corollary \ref{illustration1bis} and the form $F$ is not extraordinary.
\end{example}

However, Corollary \ref{illustration1bis} gives no information on the extraordinary form $G$ linked with $F$. This drawback is fixed in the following statements, which were announced in \cite[Theorem B]{FK1}. We designate by $[F]_{\GL}$ and $[F]_{\rm val}$ the equivalence classes of the form $F$ relatively to $\sim_{\GL}$ and $\sim_{\rm val}$. Of particular importance will be the form $F^\dag$ attached to $F$ defined by
$$
F^\dag(X, Y) := F(2X, Y).
$$

\begin{corollary}
\label{illustration2bis} 
Let $d \geq 3$. A form $F \in {\rm Bin}(d, \Z)$ satisfying $(C3)$ is extraordinary if and only if there exists $G \in [F]_{\rm val}$ such that
$$
\{\sigma \in {\rm Aut}(G, \Q) : \sigma^3 = {\rm id}\} \subseteq {\rm GL}(2, \Z) \textup{ and } 3 \mid \vert {\rm Aut}(G, \Q)\vert. 
$$
Furthermore, we have a decomposition into two disjoint $\sim_{\GL}$ classes
$$
[F]_{\rm val} = [G]_{{\rm GL}(2, \Z)} \cup [G^\dag]_{{\rm GL}(2, \Z)}.
$$
\end{corollary}

Since $F$ satisfies $(C3)$, we note that the condition $3 \mid \vert {\rm Aut}(G, \Q)\vert$ is automatically satisfied. We have opted to state the corollary in this way to keep the statement fully analogous for automorphism groups other than $\mathbf{D}_3$ or $\mathbf{D}_6$.

In fact, Corollaries \ref{illustration1bis} and \ref{illustration2bis} are consequences of the following theorem, the proof of which occupies the most important part of this paper.

\begin{theorem}
\label{sourcebis}
Let $d \geq 3$. Let $F_1 \in {\rm Bin}(d, \Q)$ be an extraordinary form satisfying $(C3)$. Let $F_2 \in {\rm Bin}(d, \Q)$ be an extraordinary form linked with $F_1$, in other words, by hypothesis we have
$$
F_1 \sim_{\rm val} F_2 \textup{ and } F_1 \not \sim_{{\rm GL}(2, \Z)} F_2.
$$
Then there exist two forms $G_1$ and $G_2$ such that $G_i \sim_{{\rm GL}(2, \Z)} F_i$ and such that
\begin{equation}
\label{G1= G2bis}
G_1^\dag = G_2\textup{ or } G_2^\dag = G_1.
\end{equation}
In the first case we have
$$
\{\sigma \in {\rm Aut}(G_1, \Q) : \sigma^3 = {\rm id}\} \subseteq {\rm GL}(2, \Z),
$$
and in the second case, we have
\begin{equation}
\label{supper}
\{\sigma \in {\rm Aut}(G_2, \Q) : \sigma^3 = {\rm id}\} \subseteq {\rm GL}(2, \Z).
\end{equation}
\end{theorem}

It is not difficult to see that the converse of Theorem \ref{sourcebis} also holds. Indeed, this is an easy consequence of Corollary \ref{illustration2bis}.

\subsection*{Acknowledgements} 
The first author thanks Michel Waldschmidt for inspiring the thema of this paper, for sharing his ideas and for his encouragements. The second author gratefully acknowledges the support of Dr. Max R\"ossler, the Walter Haefner Foundation and the ETH Z\"urich Foundation.

\section{\texorpdfstring{From Theorem \ref{sourcebis} to Corollaries \ref{illustration1bis} and \ref{illustration2bis}}{From the main theorem to the corollaries}}
\subsection{\texorpdfstring{From Corollary \ref{illustration2bis} to Corollary \ref{illustration1bis}}{Proof of Corollary 1.1}}
\label{Fromto}
The proof is exactly the same as in \cite[\S 3.2]{FK1}.

\subsection{\texorpdfstring{From Theorem \ref{sourcebis} to Corollary \ref{illustration2bis}}{Proof of Corollary 1.3}}
The proof is exactly the same as in \cite[\S 4]{FK1}.

\subsection{\texorpdfstring{Overview of Theorem \ref{sourcebis}}{Overview of Theorem 1.4}}
It remains to prove Theorem \ref{sourcebis}. Its proof follows the strategy from \cite[Theorem 1.5]{FK1} and \cite[Theorem 1.2]{FKD4}. We summarize this strategy as follows: starting from a pair of linked extraordinary forms $(F_1, F_2)$, we build two forms $G_i \in [F_i]_{\GL}$ ($i=1$ or $2$) and two integers $D$ and $\nu$ such that $G_1(X, Y) = G_2(DX, DY/\nu)$ by appealing to a result of Fouvry and Waldschmidt \cite{FW} on the geometry of binary forms (see Proposition \ref{essential}). This leads to a covering of $\Z^2$ by a certain number of lattices, the definition of which is based on the coefficients of the matrices of ${\rm Aut}(G_1, \Q)$ (but not on the coefficients of $G_1$ or $G_2$). The number of these lattices is equal to the cardinality of ${\rm Aut}(G_1, \Q)$, divided by $2$ if $d$ is even.
 
In the present work we study coverings of $\Z^2$ by six lattices. This leads to a substantial increase of complexity, and we have resorted to the help of a computer to solve it. The condition that these lattices cover $\Z^2$ is a huge constraint, leading to essentially one value for $(D, \nu)$ which is $(2, 2)$.

In the next section, we recall all the basic tools from \cite{FK1, FKD4}. However, Lemmas \ref{length5} and \ref{length6} are new.

\section{Lattices, coverings and extraordinary forms} 
\subsection{Lattices}
By a {\it lattice}, we mean an additive subgroup of $\Z^2$ with rank $2$. A lattice is {\it proper} when it is different from $\Z^2$. If $\Lambda$ is a lattice generated by linearly independent vectors $\vec u, \vec v \in \Z^2$, the {\it index} of $\Lambda$ is the positive integer
$$
[\Z^2 : \Lambda] = \vert \Z^2/\Lambda \vert = \vert \det(\vec u, \vec v) \vert.
$$

\begin{definition}
Let $\gamma \in {\rm GL}(2, \Q)$. We define the lattice $L(\gamma)$ through
$$
L(\gamma) := \left\{(x, y)\in \Z^2 : \gamma\begin{pmatrix} x \\ y \end{pmatrix} \in \Z^2\right\}.
$$
\end{definition}

\noindent We will frequently use the obvious remarks
\begin{equation}
\label{obvious0}
L(\gamma) = L(-\gamma)
\end{equation}
and
\begin{equation}
\label{obvious}
L(\gamma) = \Z^2 \iff \gamma \text{ has integer coefficients}.
\end{equation}
We recall the following property, which is proved by a direct calculation (see \cite[Lemma 6.9] {FK1}).

\begin{lemma}
\label{multipleofdet-1} 
Let $\gamma \in {\rm GL}(2, \Q)$. Then $[\Z^2 : L(\gamma)]$ is an integer multiple of $\vert \det \gamma \vert^{-1}$.
\end{lemma}

The following lemma provides a unique decomposition of matrices in ${\rm GL}(2, \Q)$ (see \cite[Lemma 6.8]{FK1}).

\begin{lemma}
\label{Standard} 
Every matrix $f \in {\rm GL}(2, \Q)$ can be uniquely written under the form
$$
f= \frac ND \begin{pmatrix} a_1& a_2\\ a_3 & a_4\end{pmatrix},
$$
where the integers $a_1$, $a_2$, $a_3$, $a_4$, $D$ and $N$ satisfy the conditions
$$
D,\, N \geq 1, \ \gcd(D,N) =1,\ \gcd(a_1,\, a_2, \, a_3, \, a_4) =1.
$$
\end{lemma}

Let $p \geq 2$ be a prime and let $v_p(n)$ be the $p$--adic valuation of the integer $n$ with the convention $v_p(0) = + \infty$. Several times we will use the following easy lemma without further reference.

\begin{lemma} 
Let $p$ be a prime. Let $\alpha, \beta, \gamma \in \Z$, with $\gamma \neq 0$. Then the lattice defined by the equation
$$
\alpha x_1 + \beta x_2 \equiv 0 \bmod \gamma
$$
\begin{itemize}
\item is included in the lattice
$$
\{(x_1, x_2) : x_1 \equiv 0 \bmod p\}
$$
as soon as the following inequality holds
$$
v_p(\alpha) <\min \left( v_p(\beta), v_p(\gamma)\right),
$$
\item has its index divisible by $p$ as soon as the following inequality holds
$$
\min \left( v_p( \alpha), v_p (\beta)\right) < v_p (\gamma).
$$
\end{itemize}
\end{lemma}

\subsection{Coverings}
\begin{definition} 
Let $k\geq 1$ be an integer and let $(\Lambda_i)_{1 \leq i \leq k}$ be $k$ lattices. We say that 
$$
\mathcal C = \{\Lambda_1, \ldots, \Lambda_k\}
$$ 
is a covering of $\Z^2$ (or a covering) if and only if we have the equality 
$$
\bigcup_{1 \leq i \leq k} \Lambda_i = \Z^2.
$$
The lattices $\Lambda_j$ are the component lattices of the covering $\mathcal C$.
\end{definition}

Here is a first property of coverings that we will use to assert that some given sets of lattices can not be coverings.

\begin{lemma}
\label{sumofdensities}
Let $k \geq 2$ and $\mathcal C = \{\Lambda_1, \ldots, \Lambda_k\}$ be a covering. We then have the inequality
$$
\sum_{i = 1}^k \frac{1}{[\Z^2 : \Lambda_i]} >1. 
$$
\end{lemma}

\begin{proof} 
We state a general combinatorial inequality: let $k\geq 2$ be an integer and let $\mathcal E_i$ (for $1 \leq i \leq k$) be subsets of some set $\mathcal E$. Let $\chi_{\mathcal F}$ be the characteristic function of a subset $\mathcal F$ of $\mathcal E$. Then for every $x\in \mathcal E$ one has the inequality
\begin{equation}
\label{sumofchi}
\chi_{\mathcal E_1 \cup \cdots \cup \mathcal E_k}(x)\leq \chi_{\mathcal E_1}(x) + \cdots + \chi_{\mathcal E_k}(x) -\chi_{\mathcal E_1 \cap \mathcal E_2}(x). 
\end{equation}
In the euclidean plane, let $D(R)$ be the disk centered at the origin and with radius $R$. By hypothesis we have $\left(\Lambda_1 \cup \dots \cup \Lambda_k \right)\cap D(R) = \Z^2 \cap D(R)$. As $R$ tends to infinity, one has
\begin{equation}
\label{insideadisk}
\sum_{\boldsymbol x\in \Z^2}\chi_{\Lambda_i \cap D(R)}(\boldsymbol x) \sim \frac{\pi R^2}{[\Z^2 : \Lambda_i]}.
\end{equation}
We apply the above formula \eqref{sumofchi} with the choice
$$
\mathcal E_i= \Lambda_i \cap D(R).
$$
We sum over all $\boldsymbol x \in \Z^2$, apply \eqref{insideadisk} and let $R$ tend to infinity to obtain the inequality
$$
\pi \leq \pi \left( \frac 1 {[\Z^2 : \Lambda_1]} + \cdots + \frac1 {[\Z^2 : \Lambda_k]}- \frac 1 {[\Z^2 : \Lambda_1\cap \Lambda_2]}\right).
$$
This gives Lemma \ref{sumofdensities}, since $\Lambda_1 \cap \Lambda_2$ is also a lattice.
\end{proof}

\begin{definition}[Minimal covering]
Let $k \geq 1$, let $\Lambda_i$ be lattices and let $\mathcal C = \{\Lambda_1, \ldots, \Lambda_k\}$ be a covering. We say that $\mathcal C$ is a minimal covering of length $k$ if and only if replacing any $\Lambda_i$ by some proper sublattice $\Lambda'_i \varsubsetneq \Lambda_i$, the set 
$$
\{\Lambda_1, \ldots, \Lambda_{i - 1}, \Lambda'_i, \Lambda_{i + 1}, \ldots, \Lambda_k \}
$$
is not a covering.
\end{definition}

If $\mathcal C$ is a minimal covering of length $k\geq 2$ and if $1 \leq i \neq j \leq k$, we never have $\Lambda_i \subseteq \Lambda_j$. In particular, every $\Lambda_i$ is a proper lattice. The following lemma asserts that, from any covering, one can extract a minimal covering (see \cite[Lemma 3.3]{FKD4}).

\begin{lemma}
\label{existenceminimal} 
Let $k \geq 1$ and let 
$$
\mathcal C := \{\Lambda_1, \ldots, \Lambda_k\}
$$
be a covering. Then there exists an integer $1 \leq k' \leq k$, an injection $\phi: \{1, \ldots, k'\} \rightarrow \{1, \ldots, k\}$ and lattices $\Lambda'_j$ ($1 \leq j \leq k'$) such that 
$$
\mathcal C ':= \{\Lambda'_1, \ldots, \Lambda'_{k'}\}
$$
is a minimal covering, and for all $1 \leq j \leq k'$ one has $\Lambda'_j \subseteq \Lambda_{\phi(j)}$. In particular, for all $1\leq j \leq k'$, one has
\begin{equation}
\label{indexdividesindex}
[\Z^2 :\Lambda_{\phi (j)} ] \,\Bigl\vert\, [\Z^2: \Lambda'_j].
\end{equation}
\end{lemma}

\begin{definition} 
With the conventions from Lemma \ref{existenceminimal} we say that $\mathcal C'$ is a minimal covering {\it extracted } from the covering $\mathcal C$. Note that, $\mathcal C$ being given, the minimal extracted covering $\mathcal C'$ is not necessarily unique.
\end{definition} 

The divisibility condition \eqref{indexdividesindex} will prove to be an efficient tool for eliminating a given set of $k$ lattices as a covering of $\Z^2$, as soon as one has at their disposal a catalog of minimal coverings with length $\leq k$. This remark is a crucial tool in the proof of Theorem \ref{sourcebis}.

\subsection{A catalog of minimal coverings} 
\label{catalog} 
If the lattice $\Lambda$ is $\Z$--generated by the two vectors $\vec u = (u_1, u_2)$ and $\vec v= (v_1, v_2)$ we write 
$$
\Lambda= 
\begin{pmatrix}
u_1 & v_1 \\ 
u_2 & v_2
\end{pmatrix}.
$$
Below we summarize the facts on minimal coverings $\mathcal C$ that we will use in our proofs.

\begin{lemma}
\label{length2}
There is no minimal covering with length $2$.
\end{lemma}

\begin{lemma}[{\cite[Lemma 6.7]{FK1}}]
\label{length3} 
There is exactly one minimal covering with length $3$, namely
$$
\left\{
\begin{pmatrix} 
2 & 0 \\ 0 & 1
\end{pmatrix},\ 
\begin{pmatrix} 
1 & 0 \\ 0 & 2
\end{pmatrix},\ 
\begin{pmatrix}
1 & 0\\ 1 & 2
\end{pmatrix}
\right\}.
$$
\end{lemma}

\begin{lemma}[{\cite[Theorem 3.4]{FKD4}}]
\label{length4}
There are exactly four minimal coverings with length $4$. These are
$$
\left\{
\begin{pmatrix}
1 & 0 \\ 
0 & 2
\end{pmatrix},
\begin{pmatrix}
4 & 0 \\ 
0 & 1
\end{pmatrix},
\begin{pmatrix}
1 & 0 \\ 
1 & 2
\end{pmatrix},
\begin{pmatrix}
2 & 0 \\ 
1 & 2
\end{pmatrix}
\right\},
$$
$$
\left\{
\begin{pmatrix}
1 & 0 \\ 
0 & 4
\end{pmatrix},
\begin{pmatrix}
2 & 0 \\ 
0 & 1
\end{pmatrix},
\begin{pmatrix}
1 & 0 \\ 
1& 2
\end{pmatrix},
\begin{pmatrix}
1 & 0 \\ 
2 & 4
\end{pmatrix}
\right\},
$$
$$
\left\{
\begin{pmatrix}
1 & 0 \\ 
0 & 2 
\end{pmatrix},
\begin{pmatrix}
2 & 0 \\ 
0 & 1
\end{pmatrix},
\begin{pmatrix}
1 & 0 \\ 
1 & 4
\end{pmatrix},
\begin{pmatrix}
1 & 0 \\ 
3 & 4 
\end{pmatrix}
\right\},
$$
$$
\left\{
\begin{pmatrix}
1 & 0 \\
0 & 3
\end{pmatrix},
\begin{pmatrix}
3 & 0 \\ 
0 & 1
\end{pmatrix},
\begin{pmatrix}
1 & 0 \\ 1 & 3
\end{pmatrix},
\begin{pmatrix}
1 & 0 \\ 
2 & 3
\end{pmatrix}
\right\}.
$$
\end{lemma}

For larger lengths, we used a computer to enumerate all the minimal coverings of length $5$ or $6$ (see the Appendix, Table \ref{table1}). From the data of this table, we extract the following information necessary to our proofs.

\begin{lemma}
\label{length5} 
There are exactly 9 minimal coverings with length equal to 5. Among these coverings:
\begin{enumerate}
\item There is no minimal covering with all the indices of the component lattices divisible by~$4$.
\item	There is no minimal covering with at least one component with index divisible by~$5$.
\end{enumerate}
\end{lemma}

\begin{lemma}
\label{length6}
There are exactly 40 minimal coverings with length equal to 6. Among these coverings, there are only two minimal coverings with all the indices of the components greater or equal to $4$. These are
\begin{align*}
\mathcal C_{4^6} &:=
\left\{
\begin{pmatrix}
1 & 0 \\
0 & 4
\end{pmatrix},
\begin{pmatrix}
4 & 0 \\ 
0 & 1
\end{pmatrix},
\begin{pmatrix}
1 & 0 \\ 
1 & 4
\end{pmatrix},
\begin{pmatrix}
1 & 0 \\ 
3 & 4
\end{pmatrix},
\begin{pmatrix}
1 & 0 \\ 
2 & 4
\end{pmatrix},
\begin{pmatrix}
2 & 0 \\ 
1 & 2
\end{pmatrix}
\right\} \\
\mathcal C_{5^6} &:=
\left\{
\begin{pmatrix}
1 & 0 \\
0 & 5
\end{pmatrix},
\begin{pmatrix}
5& 0 \\ 
0 & 1
\end{pmatrix},
\begin{pmatrix}
1 & 0 \\ 1 & 5
\end{pmatrix},
\begin{pmatrix}
1 & 0 \\ 
4 & 5
\end{pmatrix},
\begin{pmatrix}
1 & 0 \\ 
2 & 5
\end{pmatrix},
\begin{pmatrix}
1& 0 \\ 
3 & 5
\end{pmatrix}
\right\}.
\end{align*}
Furthermore, $\mathcal C_{5^6}$ is the unique minimal covering with length $6$ comprised of at least one component with index divisible by some prime $\geq 5$. 
\end{lemma}

\subsection{Links with extraordinary forms}
We recall the definition and some straightforward properties of an isomorphism between two forms. 

\begin{definition}
\label{Def2.2} 
Let $F_1$ and $F_2$ be two forms in ${\rm Bin}(d, \Q)$. An isomorphism from $F_1$ to $F_2$ is an element $\rho \in {\rm GL}(2, \Q)$ such that $F_1 \circ \rho = F_2$. The set of all such isomorphisms is denoted by ${\rm Isom}(F_1 \rightarrow F_2, \Q)$. Suppose ${\rm Isom}(F_1 \rightarrow F_2, \Q)$ is not empty and let $\rho$ be one of its elements. Then we have the equalities
\begin{align*}
{\rm Isom}(F_1 \rightarrow F_2, \Q) &= \rho \cdot {\rm Aut}(F_2, \Q) = {\rm Aut}(F_1, \Q) \cdot \rho \\
{\rm Isom}(F_2 \rightarrow F_1, \Q) &= \rho^{-1} \cdot {\rm Aut}(F_1, \Q) = {\rm Aut}(F_2, \Q) \cdot \rho^{-1}.
\end{align*}
\end{definition}

The following key proposition is \cite[Prop. 2.5]{FKD4} (see also the remarks \cite[Comments 2.6]{FKD4}). The equality \eqref{twocoverings} exhibits two coverings of $\Z^2$ that we will deeply investigate in the context of the hypothesis $(C3)$.

\begin{prop}
\label{essential} 
Let $d\geq 3$ and let $(F_1, F_2)$ be a pair of linked extraordinary forms. Then there exists $\rho \in {\rm GL}(2, \Q)$, a pair of linked extraordinary forms $(G_1, G_2)$ and a pair $(D, \nu)$ of positive integers such that
\begin{enumerate}
\item\label{467} we have $F_1 = F_2 \circ \rho$,

\item\label{468} we have
$$
\Bigl(G_1 \sim_{{\rm GL}(2,\Z)} F_1 \textup{ and } G_2 \sim_{{\rm GL}(2,\Z)} F_2\Bigr) \textup{ or } \Bigl(G_1 \sim_{{\rm GL}(2,\Z)} F_2 \textup{ and } G_2\sim_{{\rm GL}(2,\Z)} F_1\Bigr),
$$

\item we have
\begin{equation}
\label{conditionsforDandnu}
D, \nu \geq 1, D \nu > 1, D \mid \nu \textup{ and } 1 \leq \nu \leq D^2
\end{equation}
and the matrix 
\begin{equation}
\label{defgamma}
\gamma := \begin{pmatrix} D & 0\\ 0 & D/\nu
\end{pmatrix}
\end{equation}
satisfies $G_1 = G_2 \circ \gamma$,

\item we have
\begin{equation}
\label{gammaminimal}
[\Z^2 : L (\gamma)] = \min \left\{[\Z^2 : L( \tau)] : \tau \in {\rm Isom}(G_1 \rightarrow G_2, \Q) \cup {\rm Isom}(G_2 \rightarrow G_1, \Q)\right\},
\end{equation}

\item and finally, we have the two coverings
\begin{equation}
\label{twocoverings}
\Z^2 = \bigcup_{\tau \in {\rm Isom}(G_1 \rightarrow G_2, \Q)} L(\tau) = \bigcup_{\tau \in {\rm Isom}(G_2 \rightarrow G_1, \Q)} L(\tau).
\end{equation}
\end{enumerate}
\end{prop}

\noindent The definition of $\gamma$ directly implies the equality 
\begin{equation}
\label{747}
[\Z^2 : L(\gamma)] = \nu/D.
\end{equation}

\section{\texorpdfstring{Beginning of the proof of Theorem \ref{sourcebis}}{Beginning of the proof of Theorem 1.4}}
\subsection{The two coverings}
\label{2coverings}
We start from a pair of linked extraordinary forms $(F_1, F_2)$, such that ${\rm Aut}(F_1, \Q) $ is ${\rm GL}(2, \Q)$--conjugate with ${\bf D}_3$ or with ${\bf D}_6$ (condition $(C3)$). Let $(G_1, G_2)$ be a pair of linked extraordinary forms, whose existence is assured by Proposition \ref{essential}. By the conjugation formula \eqref{conjugation} and by the items \ref{467}. and \ref{468}. of Proposition \ref{essential} we deduce that
$$
{\rm Aut}(F_1, \Q), \ 
{\rm Aut}(F_2, \Q),\ {\rm Aut}(G_1, \Q) \text{ and } {\rm Aut}(G_2, \Q) \simeq_{{\rm GL}(2, \Q)} {\bf D}_3 \text{ or } {\bf D}_6.
$$
So we have
\begin{equation}
\label{twocasesbis}
{\rm Aut}(G_2, \Q) = 
\begin{cases}
T_2^{-1} {\bf D}_3T_2\\
\text { or }\\
T_2^{-1} {\bf D}_3T_2 \ \bigcup \ \left(-T_2^{-1} {\bf D}_3T_2\right),
\end{cases}
\end{equation} 
where $T_2$ is some invertible matrix where the entries $t_i$ are coprime integers
\begin{equation}
\label{defT2bis}
T_2 := \begin{pmatrix} t_1 & t_2\\ t_3 & t_4
\end{pmatrix}.
\end{equation}
 
\begin{enumerate} 
\item[--] Suppose that we are in the first case of \eqref{twocasesbis}. By the definition \eqref{defgamma}, $\gamma$ belongs to ${\rm Isom}(G_2 \rightarrow G_1, \Q)$. We appeal to Definition \ref{Def2.2} and to the explicit description \eqref{D3=...} of the group ${\bf D}_3$ in order to split the double equality of \eqref{twocoverings} into
 
 \noindent $\bullet$ the {\it main covering}
 \begin{equation}
 \label{maincovering}
 \Z^2 = \bigcup_{\sigma \in \mathfrak S} \Lambda(\sigma),
 \end{equation}
 
\noindent $\bullet$ and the {\it dual covering}
\begin{equation}
\label{dualcovering}
\Z^2 = \bigcup_{\sigma \in \mathfrak S}\Gamma (\sigma), 
\end{equation}
where 
\begin{equation}
\label{deffrakSbis}
\mathfrak S := {\rm Aut}(G_2, \Q)=
\{ {\rm id}, \, T_2^{-1} RT_2 , \, T_2^{-1} R^2T_2, , \, T_2^{-1} ST_2 , \, T_2^{-1} RST_2, \, T_2^{-1} R^2ST_2\},
\end{equation}

\begin{equation}
\label{defLambdabis}
\Lambda(\sigma) = L(\gamma^{-1} \sigma) =\{ \boldsymbol x \in \Z^2 : \gamma^{-1} \sigma (\boldsymbol x )\in \Z^2\},
\end{equation}
and
\begin{equation}
\label{defGammabis}
\Gamma (\sigma) = L(\sigma \gamma)=\{ \boldsymbol x \in \Z^2 : \sigma\gamma (\boldsymbol x )\in \Z^2\}.
\end{equation}

\item[--] Suppose that we are in the second case of \eqref{twocasesbis}. Then the equalities \eqref{twocoverings} will lead to two coverings of twelve lattices each. But the equality \eqref{obvious0} shows that each lattice appears twice, thus giving coverings consisting of six lattices. In conclusion, the equalities \eqref{maincovering} and \eqref{dualcovering} are true as soon as any case of \eqref{twocasesbis} holds.
\end{enumerate}

\noindent By \eqref{gammaminimal} and \eqref{747}, we have 
\begin{equation}
\label{798}
[\Z^2: \Lambda(\sigma)], \ [\Z^2 : \Gamma (\sigma)] \geq \nu/D \ (\sigma \in \mathfrak S).
\end{equation}
A useful lemma is the following

\begin{lemma}
\label{atmostfive} 
Suppose that the coverings \eqref{maincovering} and \eqref{dualcovering} exist. Then we have the inequality 
$$
\nu/D \leq 5.
$$
\end{lemma}

\begin{proof} 
Indeed, these two coverings have six components. By \eqref{798}, each component has index $\geq \nu/D$, which is an integer. By Lemma \ref{sumofdensities}
we have $6D/\nu >1$ and we complete the proof.
\end{proof}

\subsection{A particular case already treated}
Suppose that
\[
\bigcup_{\sigma \in \{\text{id}, T_2^{-1}RT_2, T_2^{-1}R^2T_2\}} \Lambda(\sigma) = \Z^2.
\]
This situation has already been treated in \cite[\S 11]{FK1} when the automorphism group is ${\bf C}_3$ or ${\bf C}_6$. It leads to the values $D = 2$ and $\nu = 2$ and ${\rm Aut}(G_2, \Q) \subseteq \GL$ (see \cite[(11.1)]{FK1}). Hence Theorem \ref{sourcebis} is proved in that particular case. So, henceforth, we assume that 
\begin{equation}
\label{backtoC3}
\bigcup_{\sigma \in \{\text{id}, T_2^{-1}RT_2, T_2^{-1}R^2T_2\}} \Lambda(\sigma) \neq \Z^2.
\end{equation}

\subsection{Proper lattices}
We will first prove that the covering \eqref{maincovering} is never the trivial covering, by which we mean that one of the component lattices equals $\Z^2$.

\begin{lemma} 
We adopt the hypotheses of Proposition \ref{essential} and the notations of \S \ref{2coverings}. Then for every $\sigma\in \mathfrak S$, the lattice $\Lambda(\sigma)$ is proper.
\end{lemma}

\begin{proof} 
We give a proof by contradiction. Suppose that there exists some $\sigma \in \mathfrak S$ satisfying
$$
[\Z^2 : \Lambda(\sigma)] = [\Z^2 : L(\gamma^{-1}\sigma)] = 1.
$$
By \eqref{obvious}, we deduce that $\gamma^{-1} \sigma$ has integral coefficients. By \eqref{gammaminimal}, we have $[\Z^2 : L(\gamma)] = 1$, thus $\gamma$ also has integral coefficients. Consider the equality $\sigma = \gamma \cdot (\gamma^{-1} \cdot \sigma)$. Since $\det \sigma =\pm 1$ and since $\det \gamma$ and $\det (\gamma^{-1}\sigma)$ are integers, we have $\det \gamma =\pm 1$. But then $G_1$ and $G_2$ are $\GL$--equivalent, contradiction.
\end{proof}

The strategy now will be to exploit the main covering \eqref{maincovering} as long as possible to restrict the possible values of $D$ and $\nu$. When we are at an impasse, we will study the dual covering \eqref{dualcovering}. One indication of the additional complexities of the present paper is by comparison with \cite{FK1} and \cite{FKD4}, where we did not investigate the dual covering.

Throughout the paper, we define the positive integer $d_2$ by
\begin{equation}
\label{defd2}
d_2 =\vert \det T_2\vert= \vert t_1t_4 -t_2t_3\vert,
\end{equation}
where $T_2$ is the conjugation matrix defined in \eqref{twocasesbis} and \eqref{defT2bis}. We also introduce the numbers
\begin{equation}
\label{defgij}
g_{1, 3} := \gcd (t_1, t_3), \quad g_{2, 4} := \gcd (t_2, t_4).
\end{equation}
They satisfy the relations
\begin{equation}
\label{propgij}
\gcd (g_{1, 3}, g_{2, 4}) = 1 \quad \text{ and } \quad g_{1, 3} g_{2, 4} \mid d_2.
\end{equation}

\subsection{\texorpdfstring{Description of the lattices $\Lambda$ by their equations}{Description of the lattices Lambda}}
Recalling the notations \eqref{defgamma}, \eqref{deffrakSbis} and \eqref{defLambdabis} we have

\begin{lemma}
\label{lLatticeEquationsbis} 
The lattices $\Lambda(\sigma)$ for $\sigma \in \mathfrak S$ are defined by the following systems of equations with $\boldsymbol x = (x_1, x_2) \in \Z^2$:
\begin{equation}
\label{Idbis}
\Lambda(\textup{id}):
\begin{cases}
x_1 &\equiv 0 \bmod D \\
\nu x_2 &\equiv 0 \bmod D,
\end{cases}
\end{equation}
\begin{equation}
\label{Rbis}
\Lambda(T_2^{-1}RT_2):
\begin{cases}
(t_1t_2 + t_2t_3 + t_3t_4) x_1 + (t_2^2 + t_2t_4 + t_4^2) x_2 &\equiv 0 \bmod d_2 D \\
\nu ((t_1^2 + t_1t_3 + t_3^2) x_1 + (t_1t_2 + t_1t_4 + t_3t_4) x_2) &\equiv 0 \bmod d_2 D,
\end{cases}
\end{equation}
\begin{equation}
\label{R2bis}
\Lambda(T_2^{-1}R^2T_2) :
\begin{cases}
(t_1t_2 + t_1t_4 + t_3t_4) x_1 + (t_2^2 + t_2t_4 + t_4^2) x_2 &\equiv 0 \bmod d_2 D \\
\nu ((t_1^2 + t_1t_3 + t_3^2) x_1 + (t_1t_2 + t_2t_3 + t_3t_4) x_2) &\equiv 0 \bmod d_2 D,
\end{cases}
\end{equation}
\begin{equation}
\label{Sbis}
\Lambda(T_2^{-1}ST_2):
\begin{cases}
(t_3t_4 - t_1t_2) x_1 + (t_4^2 - t_2^2) x_2 &\equiv 0 \bmod d_2 D \\
\nu ((t_3^2 - t_1^2) x_1 + (t_3t_4 - t_1t_2) x_2) &\equiv 0 \bmod d_2 D,
\end{cases}
\end{equation}
\begin{equation}
\label{RSbis}
\Lambda(T_2^{-1}RST_2):
\begin{cases}
(t_1t_2 + t_1t_4 + t_2t_3) x_1 + (t_2^2 + 2t_2t_4) x_2 &\equiv 0 \bmod d_2 D \\
\nu ((t_1^2 + 2t_1t_3) x_1 + (t_1t_2 + t_1t_4 + t_2t_3) x_2) &\equiv 0 \bmod d_2 D,
\end{cases}
\end{equation}
and finally
\begin{equation}
\label{R2Sbis}
\Lambda(T_2^{-1}R^2ST_2):
\begin{cases}
(t_1t_4 + t_2t_3 + t_3t_4) x_1 + (t_4^2 + 2t_2t_4) x_2 &\equiv 0 \bmod d_2 D \\
\nu ((t_3^2 + 2t_1t_3) x_1 + (t_1t_4 + t_2t_3 + t_3t_4) x_2) &\equiv 0 \bmod d_2 D.
\end{cases}
\end{equation}
\end{lemma}

\begin{proof}
This follows from a direct computation.
\end{proof}

\begin{remark} 
\label{rLidSimple}
The condition $D \mid \nu$ appearing in \eqref{conditionsforDandnu} enables us to simplify these equations: we may delete the second equation of \eqref{Idbis} to obtain the equality $\Lambda({\rm id}) = \Lambda_D$ with 
$$
\Lambda_D := \{(x_1, x_2): x_1 \equiv 0 \bmod D\}.
$$
\end{remark}

To shorten notations, we use the symbol $\Lambda^?$ to denote an unspecified proper lattice (this symbol may denote different lattices at each occurrence). We also write
$$
{\mathfrak S}^\flat := \{\sigma \in \mathfrak S : \sigma \neq{\rm id}\} = \{T_2^{-1} RT_2 , \, T_2^{-1} R^2T_2, \, T_2^{-1} ST_2 , \, T_2^{-1} RST_2, \, T_2^{-1} R^2ST_2\}.
$$
Similarly, for $2 \leq k \leq 5$, we denote by
\begin{equation}
\label{defS<k}
\mathfrak S^\flat_{\leq k}
\end{equation}
some unspecified subset of $\mathfrak S^\flat$ with cardinality $\leq k$.

For $\sigma\in \mathfrak S^\flat$, we write each of these systems $\eqref{Rbis}, \ldots, \eqref{R2Sbis}$ in the form
\begin{equation}
\label{generalnotation}
\Lambda(\sigma):
\begin{cases}
p_{1, \sigma}({\bf t}) x_1 + p_{2, \sigma}({\bf t}) x_2 &\equiv 0 \bmod d_2D,\\
\nu \left(P_{1, \sigma}({\bf t}) x_1 + P_{2, \sigma}({\bf t}) x_2 \right) &\equiv 0 \bmod d_2D,
\end{cases}
\end{equation}
where ${\bf t} = (t_1, t_2, t_3, t_4)$. Notice that $p_{2, \sigma}({\bf t})$ is a polynomial in $t_2$ and $t_4$ only and that $P_{1, \sigma}({\bf t})$ is a polynomial in $t_1$ and $t_3$ only. Furthermore, we put
\begin{equation}
\label{generalnotation*}
\left(p_{1,{\rm id}}({\bf t}) , p_{2, {\rm id}}({\bf t})\right) = (1, 0).
\end{equation}

\begin{remark}
\label{representativematrix} 
The representative matrix ${\rm Mat}(\sigma)$ of the morphism $\sigma$, for $\sigma \in \mathfrak S^\flat$, is
$$
{\rm Mat}(\sigma) = \pm
\begin{pmatrix}
p_{1, \sigma}({\bf t})/d_2 & p_{2, \sigma}({\bf t})/d_2\\
P_{1, \sigma}({\bf t})/d_2 & P_{2, \sigma}({\bf t})/d_2
\end{pmatrix}. 
$$
In particular, ${\rm Mat}(\sigma)$ belongs to ${\rm GL}(2,\Z)$ if and only if $d_2$ divides $p_{i, \sigma}({\bf t})$ and $P_{i, \sigma}({\bf t})$ for $i \in \{1, 2\}$.
\end{remark}

\begin{remark}
\label{930} 
Fix $\nu =1$. Then each of the six systems $\eqref{Idbis}, \ldots \eqref{R2Sbis}$ defining the lattices $\Lambda(\sigma)$ benefit from the following symmetry: consider the change of parameters and variables
$$
\Theta : (t_1, \, t_2, \, t_3, \, t_4, \, x_1, \, x_2) \longleftrightarrow (t_2, \, t_1, \, t_4, \, t_3, \, x_2, \, x_1). 
$$
Then $\Theta$ permutes the two equations of each system.
\end{remark}

We will frequently work with the reduced variable $\tilde{\bf t}$ defined by the equality 
\begin{equation}
\label{defttilde}
\tilde{\bf t} := (t_1, \tilde t_2, t_3, \tilde t_4), \text{ with }
\tilde t_i = t_i/g_{2, 4} \ (i = 2,\, 4).
\end{equation}
By homogeneity, the first equation of \eqref{generalnotation} is equivalent to 
$$
p_{1, \sigma}(\tilde{\bf t}) x_1 + g_{2, 4}\,p_{2, \sigma}(\tilde{\bf t}) x_2 \equiv 0 \bmod d_2D/g_{2, 4}.
$$
For any divisor $\ell$ of $d_2D/g_{2, 4}$ we introduce the lattice $\tilde{\Lambda}^{(\ell)}(\sigma)$ defined by 
\begin{equation}
\label{954}
\tilde{\Lambda}^{(\ell)}(\sigma) : p_{1, \sigma}(\tilde{\bf t}) x_1 + g_{2, 4}\,p_{2, \sigma}(\tilde{\bf t}) x_2 \equiv 0 \bmod \ell.
\end{equation}
This lattice is not necessarily proper, but it satisfies the relation
\begin{equation}
\label{958}
\Lambda(\sigma) \subseteq \tilde{\Lambda}^{(\ell)}(\sigma)\ (\sigma \in \mathfrak S^\flat \text{ and } \ell \mid d_2D/g_{2, 4}).
\end{equation}
For convenience, we define 
\begin{equation}
\label{962}
\tilde{\Lambda}^{(\ell)}({\rm id} ) = \Lambda_{\ell}\ \text{ if }\ell \mid D,
\end{equation}
so that \eqref{958} is also satisfied. This inclusion will be frequently used in our proof, so we call $\tilde{\Lambda}^{(\ell)}(\sigma)$ the {\it $\ell$--enveloping lattice associated with $\sigma$}.

\subsection{\texorpdfstring{The determinant $D^2/\nu$ and its properties}{The determinant and its properties}}
By the definition \eqref{deffrakSbis} of $\mathfrak S$, the definition \eqref{defgamma} of $\gamma$, we have, for every $\sigma \in \mathfrak S$, the equality
\begin{equation}
\label{D2/nu=a/b}
\vert\, \det (\gamma^{-1} \sigma) \, \vert = \frac{\nu}{D^2} = \frac{b}{a},
\end{equation}
where $a$ and $b$ are integers satisfying the conditions
\begin{equation}
\label{condforaandbbis}
\gcd(a, b) =1, \quad 1\leq b \leq a.
\end{equation}
The last inequality is a consequence of \eqref{conditionsforDandnu}. 

\begin{remark}
\label{useful} 
As a consequence of Lemma \ref{multipleofdet-1}, the index of every component lattice $\Lambda(\sigma)$ of the main covering \eqref{maincovering} is an integer divisible by $a$.
\end{remark} 

Little by little, we will restrict the possible values for the integers $a$, $b$, $D$, $\nu$ and $d_2$. The {\it steps} of these successive restrictions are based on more and more intricate considerations on the index and on the equations of the components of the coverings \eqref{maincovering} and \eqref{dualcovering}. The first one is 

\begin{step}
\label{1<b<a<6} 
With the above notations and hypotheses, one has the inequalities
$$
1\leq b \leq a \leq 5.
$$
\end{step}

\begin{proof} 
The six lattices $\Lambda(\sigma)$ are a covering of $\Z^2$ by \eqref{maincovering}. By Remark \ref{useful}, for each $\sigma\in \mathfrak S$
the index $[\Z^2 : \Lambda(\sigma)]$ is an integer multiple of $a$, so its value is at least $ a $. Then Lemma~\ref{sumofdensities} leads to the inequality
$6/ a >1$, so $ a \leq 5.$
\end{proof}

We will now discuss the intimate properties of the coefficients defining the first equations of the lattices $\Lambda(\sigma)$ (see \eqref{generalnotation}). 

\subsection{\texorpdfstring{Congruence properties of the coefficients defining the lattices $\Lambda$}{Congruence properties of the coefficients defining the lattices Lambda}} 
We want to extract some structure in the equations defining the lattices $\Lambda(\sigma)$.

\begin{mydef}
\label{dZn} 
We adopt the notations from equations \eqref{generalnotation} and \eqref{generalnotation*}. For $\sigma \in \mathfrak S$, for $\mathbf{t} = (t_1, t_2, t_3, t_4)$ a $4$--tuple of coprime integers with $d_2 \neq 0$ and for $x$ an integer $\geq 1$, we denote by $Z_{\mathbf{t}, \textup{top}}(x)$ the number of $\sigma \in \mathfrak S$ for which the point
\[
(p_{1, \sigma}(\mathbf{t}), p_{2, \sigma}(\mathbf{t}))
\]
has order less than $x$ in the additive group $(\Z/x\Z)^2$. 

Now consider the set $S = S_{\mathbf t, \textup{top}}(x)$ of $\sigma \in \mathfrak S$ for which
\[
(p_{1, \sigma}(\mathbf{t}), p_{2, \sigma}(\mathbf{t}))
\]
has order $x$ in $(\Z/x\Z)^2$. We say that $\sigma \sim \sigma'$ if there exists an invertible element $ \alpha \in (\Z/x\Z)^{\ast }$ such that
\[
\alpha \cdot (p_{1, \sigma}(\mathbf{t}), p_{2, \sigma}(\mathbf{t})) = (\alpha p_{1, \sigma}(\mathbf{t}), \alpha p_{2, \sigma}(\mathbf{t})) = (p_{1, \sigma'}(\mathbf{t}), p_{2, \sigma'}(\mathbf{t})).
\]
We write $n_{\mathbf{t}, \textup{top}}(x)$ for the number of equivalence classes in $S$.
\end{mydef}

\begin{remark} 
\label{equalityoflattices}
Let $x \geq 1$ and let $\sigma$ and $\sigma'$ be two elements from $S_{\mathbf t, \textup{top}}(x)$ such that $\sigma \sim \sigma'$. Then, for every integer $g$ the two lattices respectively defined by the equations
$$ 
p_{1, \sigma}(\mathbf{t}) x_1+ g \cdot p_{2, \sigma}(\mathbf{t})x_2 \equiv 0 \bmod x
$$
and
$$
p_{1, \sigma'}(\mathbf{t}) x_1+ g \cdot p_{2, \sigma'}(\mathbf{t})x_2 \equiv 0 \bmod x
$$
are equal. We will use this remark several times below to delete some redundant lattices appearing in a covering of $\Z^2$.
\end{remark}


\begin{remark}
By the definition \eqref{generalnotation*}, we observe that $(1, 0) \in S_{\mathbf{t}, \textup{top}}(x)$. So $S_{\mathbf{t}, \textup{top}}(x)$ is never empty and we have $ n_{\mathbf t, \textup{top}}(x)\geq 1$ and $Z_{\mathbf{t}, \textup{top}}(x) \leq 5$.
\end{remark}

From the equality
$$
Z_{\mathbf{t}, \textup{top}}(x) + \vert S_{\mathbf t, \textup{top}}(x)\vert =\vert \mathfrak S \vert = 6,
$$
we deduce the inequality
$$
Z_{\mathbf{t}, \textup{top}}(x) + n_{\mathbf t, \textup{top}}(x) \leq 6.
$$
The following lemma improves on this inequality for small values of the integer $x$. We have

\begin{lemma}
\label{lZxnx}
Let $\mathbf{t} = (t_1, t_2, t_3, t_4) \in \Z^4$ with $\gcd(t_2, t_4) = 1$. 

\begin{itemize}
\item $(x = 5)$. We have the inequalities 
\begin{align}
\label{eZxnx45bis}
Z_{\mathbf{t}, \textup{top}}(5) + n_{\mathbf{t}, \textup{top}}(5) \leq 3, \quad Z_{\mathbf{t}, \textup{top}}(5) \leq 1.
\end{align}

\item $(x = 4)$. We have the inequalities
\begin{equation}
\label{2795}
Z_{\mathbf{t}, \textup{top}}(4) + n_{\mathbf{t}, \textup{top}}(4) \leq 4, \quad Z_{\mathbf{t}, \textup{top}}(4) \leq 1.
\end{equation}

\begin{itemize}
\item[(i)] The inequalities \eqref{2795} become both equalities only if there exists an element $\sigma \in \mathfrak{S}$ such that $(p_{1, \sigma}(\mathbf{t}), p_{2, \sigma}(\mathbf{t})) \equiv (2, 0) \bmod 4$.

\item[(ii)] Additionally suppose that $\gcd(t_1, t_3, 2) = 1$. Then 
$$
\{(0, 1), (0, 3), (2, 1), (2, 3)\} \not \subseteq \left\{ \left(\alpha \,p_{1, \sigma }(\mathbf t), \alpha\, p_{2, \sigma}(\mathbf t)\right) : 
\alpha \in (\Z /4\Z)^*, \sigma \in S_{\mathbf t, {\rm top} }(4) \right\}.
$$
\item[(iii)] If $\gcd(t_1, t_3, 2) = 1$ and furthermore $Z_{\mathbf{t}, \textup{top}}(4) = 1$, then
$$
\{(0, 1), (0, 3), (2, 1), (2, 3)\} \cap \left\{\left(\alpha p_{1, \sigma }(\mathbf t), \alpha p_{2, \sigma}(\mathbf t)\right) : 
\alpha \in (\Z/4\Z)^\ast, \sigma \in S_{\mathbf t, {\rm top}}(4)\right\} = \varnothing.
$$
\end{itemize}

\item $(x = 3)$. We have the inequalities
\begin{equation}
\label{3335}
Z_{\mathbf{t}, \textup{top}}(3) + n_{\mathbf{t}, \textup{top}}(3) \leq 3 \textup{ and } Z_{\mathbf{t}, \textup{top}}(3) \leq 1, \textup{ or } Z_{\mathbf{t}, \textup{top}}(3) = 5.
\end{equation}
-- We have $Z_{\mathbf{t}, \textup{top}}(3) = 5$ if and only if
\begin{multline}
\label{eBadt}
(t_1, t_2, t_3, t_4) \bmod 3 \in \\ 
\mathcal T_3 :=\left\{(0, 1, 0, 1), (0, 2, 0, 2), (1, 1, 1, 1), (2, 2, 2, 2), (1, 2, 1, 2), (2, 1, 2, 1)\right\}.
\end{multline}
In particular, we have the inequality $v_3(t_1t_4-t_2t_3)\geq 1$.
\end{itemize}
\end{lemma}

\begin{proof}
This follows from a direct computation, see Subsection \ref{aBruteForce}.
\end{proof} 

The case $x = 9$ requires some variation that we describe now. Let $\mathcal T_3$ be the set of $4$--tuples $\mathbf t$ appearing on the right--hand side of \eqref{eBadt}. So we have 
\[
p_{1, \sigma}({\bf t}) \equiv p_{2, \sigma}({\bf t}) \equiv 0 \bmod 3
\]
for all ${\bf t} \in \mathcal T_3$ and for all $\sigma\in \mathfrak S^\flat$. We are now searching for information about these coefficients modulo $9$. This is the purpose of Lemma \ref{messfor9} below. For $\mathbf t \in \mathcal T_3$, for $\sigma \in \mathfrak S^\flat$ and for $i \in \{1, 2\}$ we consider the integers $\widetilde{p_{i, \sigma}}({\mathbf t}) := p_{i, \sigma}(\mathbf t)/3$. By convention, we define $(\widetilde{p_{1, {\rm id}}}(\mathbf t), \widetilde{p_{2, {\rm id}}}(\mathbf t)) = (1, 0)$. Let $\widetilde{Z_{{\mathbf t}, {\rm top}}}(9)$ be the number of $\sigma \in {\mathfrak S}$ such that
$$
(\widetilde{p_{1, \sigma}}(\mathbf t), \widetilde{p_{2, \sigma}}(\mathbf t)) \equiv (0, 0) \bmod 3. 
$$
On the set 
$$
\mathcal E_9(\mathbf t) := \{\sigma \in \mathfrak S : ( \widetilde{p_{1, \sigma} }(\mathbf t), \widetilde{p_{2, \sigma} }(\mathbf t)) \not \equiv (0, 0) \bmod 3\}, 
$$
we define the following equivalence relation: we set $\sigma \sim \sigma'$ if and only if there exists $\alpha \not \equiv 0 \bmod 3$ such that 
$$
(\widetilde{p_{1, \sigma} }(\mathbf t), \widetilde{p_{2, \sigma} }(\mathbf t)) = (\alpha\, \widetilde{p_{1, \sigma'}}(\mathbf t), \alpha \, \widetilde{p_{2, \sigma'}}(\mathbf t)).
$$
Then we denote by $\widetilde{n_{{\bf t}, {\rm top}}}(9)$ the number of equivalence classes in $\mathcal E_9(\mathbf t)$ under the equivalence relation $\sim$. We now state
 
\begin{lemma}
\label{messfor9} 
We adopt the notations just above. Let $\mathbf t = (t_1, t_2, t_3, t_4) \in \Z^4$ such that $\gcd(t_2, t_4) =1$ and such that $\boldsymbol t \in \mathcal T_3$ . Then we have the inequalities 
\begin{equation}
\label{1914}
\widetilde{Z_{\mathbf{t}, \textup{top}}}(9) + \widetilde{ n_{\mathbf{t}, \textup{top}}}(9) \leq 3 \textup{ and } \widetilde{Z_{\mathbf{t}, \textup{top}}}(9) \leq 1 .
\end{equation}
\end{lemma}

\begin{proof} 
This also follows from a direct computation, see Subsection \ref{aBruteForce}.
\end{proof}

It is perhaps worthwhile to point out the analogy between the inequality \eqref{1914} and the first part of the inequality \eqref{3335}. We announce one more property of $p_{1, \sigma}$ before proceeding.

\begin{lemma}
\label{lTripleVanishing}
Let $p \neq 3$ be a prime number or let $p = 9$. Let $(t_1, t_2, t_3, t_4) \in \Z^4$ be such that $\gcd(t_2, t_4) = \gcd(t_1, t_3) = 1$. Then we have
\[
|\{\sigma \in \mathfrak{S}^\flat : p_{1, \sigma}({\bf t}) \equiv 0 \bmod p\}| \leq 3
\]
with equality only possible for the set $\{T_2^{-1} S T_2, T_2^{-1} RS T_2, T_2^{-1} R^2S T_2\}$ and $p \equiv 1 \bmod 3$.

Furthermore, we have
\[
|\{\sigma \in \mathfrak{S}^\flat : p_{1, \sigma}({\bf t}) \equiv 0 \bmod 3\}| \in \{0, 1, 2, 5\}.
\]
If the above cardinality equals $5$, then $\mathbf{t}$ lies in the set 
$$
\mathcal{T}_3^\ast := \{(1, 1, 1, 1), (2, 2, 2, 2), (1, 2, 1, 2), (2, 1, 2, 1)\}.
$$
\end{lemma}

\begin{proof}
It is possible to prove this without using the aid of a computer algebra package. However, to save some effort, we have decided to use SageMath. Let us suppose that
\begin{equation*}
p_{1, T_2^{-1}RT_2}({\bf t}) \equiv p_{1, T_2^{-1}R^2T_2}({\bf t}) \equiv p_{1, T_2^{-1}ST_2}({\bf t}) \equiv 0 \bmod p.
\end{equation*}
For the other combinations of polynomials from Lemma \ref{lLatticeEquationsbis}, we refer the reader to Subsection \ref{aTripleVanishing}. Then we run the following script.

\begin{verbatim}
R.<t1, t2, t3, t4, u, v, w, x> = ZZ['t1, t2, t3, t4, u, v, w, x']
I = ideal(t1 * t2 + t2 * t3 + t3 * t4, t1 * t2 + t1 * t4 + t3 * t4, 
t3 * t4 - t1 * t2, t2 * u + t4 * v - 1, t1 * w + t3 * x - 1)
B = I.groebner_basis()
\end{verbatim}

The first three elements in the ideal are the polynomials from Lemma \ref{lLatticeEquationsbis}, specifically from the top of equations (\ref{Rbis}), (\ref{R2bis}), (\ref{Sbis}), while the last equation encodes the coprimality conditions $\gcd(t_2, t_4) = \gcd(t_1, t_3) = 1$. SageMath shows that $3$ is always contained in the Gr\"obner basis. This implies the first part of the lemma.

The second part of the lemma can be checked directly by hand or by the brute force computations in Subsection \ref{aBruteForce}, specifically \ref{a414}.
\end{proof}

Since the conclusion of Lemma \ref{lTripleVanishing} depends only on $t_i$ modulo $p$, it is straightforward to slightly weaken the hypotheses of this lemma to $\gcd(t_1, t_3, p) = 1$. We shall often use the lemma this way without further comment.

Concerning the polynomial $P_{1, \sigma}$ we will use the following

\begin{lemma}
\label{3467*}
Consider the set $\mathcal S := \{A, B, C, D\}$ of four polynomials in two variables $U$ and $V$ given by $A(U, V) = U^2+UV+V^2$, $B(U, V) = V^2 - U^2$, $C(U, V) = U^2 + 2UV$, $D(U, V) = V^2 + 2UV$. Let $(u, v) \in \Z^2$ be such that $\gcd(u, v, 3) = 1$. We then have
$$
A(u, v) \not\equiv 0 \bmod 9
$$
and 
\begin{equation}
\label{3469} 
\left\vert \{P\in \mathcal S : P(u, v)\equiv 0 \bmod 9\}\right\vert \leq 1.
\end{equation}
\end{lemma}

\begin{proof} 
We make a succession of simple observations:
\begin{itemize}
\item For any integers $(u, v)$, we have $A(u, v)\equiv 0 \bmod 3$ if and only if $u \equiv v \bmod 3$. If $u \equiv v \bmod 3$, then $A(u, v) \equiv 3u^2 \bmod 9$. We deduce that $\gcd(u, v, 3) = 1$ implies $A(u, v) \not\equiv 0 \bmod 9$.
\item Since $B = D - C$, in order to prove the inequality \eqref{3469}, it is sufficient to show that the polynomials $C$ and $D$ do not simultaneously vanish modulo $9$, on some $(u, v) \in \Z^2$ with $\gcd(u, v, 3) = 1$.
\item Let $(u, v)$ be such that $C(u, v) \equiv D(u, v) \equiv 0 \bmod 9$. By taking the difference, we obtain that $(u - v)(u + v)\equiv 0 \bmod 9$. We analyze this congruence in three subcases
\begin{itemize}
\item if $u \equiv v \bmod 9$, then $C(u, v) \equiv 3 u^2 \bmod 9$, and the condition $C(u, v) \equiv 0 \bmod 9$ implies $u \equiv 0 \bmod 3$, so $v \equiv 0 \bmod 3$,
\item if $u \equiv -v \bmod 9$, then $C(u, v) \equiv -u^2 \bmod 9$ and we obtain once again $u \equiv v \equiv 0 \bmod 3$,
\item if $u \equiv v \bmod 3$ and $u \equiv -v \bmod 3$, then $u \equiv v \equiv 0 \bmod 3$.
\end{itemize}
\end{itemize}
This completes the proof of the lemma.
\end{proof}

Applying this lemma yields the inequality
\begin{align}
\label{e3Pbound*}
|\{\sigma \in \mathfrak{S}^\flat : P_{1, \sigma}({\bf t}) \equiv 0 \bmod 9\}| \leq 1.
\end{align}

\subsection{\texorpdfstring{The indicator $g_2$}{The indicator g2}}
\label{indicator} 
We want to measure the discrepancy between two natural ways of expressing the matrix representing $\sigma$. Given $\sigma \in \mathfrak{S}$, we define $d_2(\sigma) \in \Z_{>0}$ by writing
\begin{equation}
\label{2895}
\sigma = \frac{1}{d_2(\sigma)}
\begin{pmatrix}
a & b \\
c & d
\end{pmatrix}
\end{equation}
in its minimal form, which is possible thanks to Lemma \ref{Standard} and the equality $\det \sigma = \pm 1$. In particular, we have $\gcd(a, b, c, d) = 1$. Furthermore, by \eqref{deffrakSbis} and by the definition of $T_2$ (see \eqref{defT2bis}) we can also write
$$
\sigma = \frac{1}{d_2} \begin{pmatrix} a' & b' \\ c' & d' \end{pmatrix}
$$
with integers $a'$, $b'$, $c'$ and $d'$ and $d_2$ defined in \eqref{defd2}. Comparing this expression with \eqref{2895}, we deduce that $d_2(\sigma) \mid d_2$, which allows us to introduce the integer $g_2(\sigma) := d_2/d_2(\sigma)$. By definition, we have
\begin{equation}
\label{g2=gcd}
g_2(\sigma) = \gcd(a', b', c', d').
\end{equation}

\begin{lemma}
\label{lSX3.2}
Let $\sigma$ and $\sigma' $ be two distinct elements in $\mathfrak{S}$ and suppose that $\sigma, \sigma'$ do not both have order $3$. Then we have
\[
\gcd(g_2(\sigma), g_2(\sigma')) \in \{1, 3\}.
\]
\end{lemma}

\begin{proof}
We will only consider the case $\sigma = T_2^{-1}R^2T_2$ and $\sigma' = T_2^{-1}ST_2$. For the other cases, we refer the reader to Subsection \ref{aSX3.2}. In this case, we run the following script.

\begin{verbatim}
R.<t1, t2, t3, t4, u1, u2, u3, u4> = ZZ['t1, t2, t3, t4, u1, u2, u3, u4']
I = ideal(t1 * t2 + t1 * t4 + t3 * t4, t2 * t2 + t2 * t4 + t4 * t4, 
t1 * t1 + t1 * t3 + t3 * t3, t1 * t2 + t2 * t3 + t3 * t4, t3 * t4 - t1 * t2, 
t4 * t4 - t2 * t2, t1 * t1 - t3 * t3, t1 * t2 - t3 * t4, 
u1 * t1 + u2 * t2 + u3 * t3 + u4 * t4 - 1)
B = I.groebner_basis()
\end{verbatim}

Here we include all the polynomials in the defining equations in Lemma \ref{lLatticeEquationsbis}, so equation (\ref{R2bis}) and equation (\ref{Sbis}) in our specific case. We have also included the condition $\gcd(t_1, t_2, t_3, t_4) = 1$ in the last equation. SageMath gives that $3$ is always contained in the Gr\"obner basis, which readily implies the lemma. 
\end{proof}

\section{\texorpdfstring{Restricting the possible values for $a$}{Restricting the possible values for a}}
We want to tighten the interval $1\leq a \leq 5$ appearing in Step \ref{1<b<a<6}. 

\subsection{\texorpdfstring{The case $a = 5$ is impossible}{The case a = 5 is impossible}}
\label{a=5isimpossible} 
We will improve Step \ref{1<b<a<6} into

\begin{step}
\label{1<b<a<5} 
With the notations and hypotheses of Step \ref{1<b<a<6}, one has the inequalities
$$
1 \leq b \leq a \leq 4.
$$
\end{step}

We give a proof by contradiction. So we suppose that $a = 5$. Recall the definitions of $g_{1, 3}$ and $g_{2, 4}$ from \eqref{defgij}. Then \eqref{D2/nu=a/b} and \eqref{condforaandbbis} imply
$$
\frac{D^2}{\nu} = \frac{a}{b} = \frac{5}{b} \quad (5 \nmid b),
$$
so we know that $5$ divides $D$. For every $\sigma \in \mathfrak S$, one has 
\begin{equation}
\label{5dividesindex}
5\mid [\Z^2: \Lambda(\sigma)]
\end{equation}
by Remark \ref{useful}. By Lemmas \ref{length3}, \ref{length4}, \ref{length5} and \ref{length6}, we deduce that the main covering $(\Lambda(\sigma))_{\sigma \in \mathfrak S}$ necessarily coincides with the covering $\mathcal C_{5^6}$. In particular, all the lattices have index $5$ but we will not use this equality in the proof. So far, we know that
$$
v_5(D) \geq 1 \text{ and } v_5(\nu) = 2v_5(D) - 1.
$$
In all cases we have $5 \mid d_2 D/g_{2, 4}$, so we are allowed to consider the $5$--enveloping lattices $ \tilde{\Lambda}^{(5)}(\sigma)$ as defined in \eqref{defttilde}, \eqref{954}, \eqref{958} and \eqref{962}. In particular, we have the inclusion
\begin{equation}
\label{1300}
\Lambda(\sigma) \subseteq \tilde{\Lambda}^{(5)}(\sigma) \quad (\sigma \in \mathfrak S^\flat)
\end{equation}
with
\begin{equation}
\label{1283bis}
\tilde{\Lambda}^{(5)}(\sigma) : p_{1, \sigma}(\tilde{\bf t}) x_1 + g_{2, 4} p_{2, \sigma}(\tilde{\bf t}) x_2 \equiv 0 \bmod 5 \quad (\sigma \in \mathfrak S^\flat),
\end{equation}
and
\begin{equation}
\label{1308}
\tilde{\Lambda}^{(5)}({\rm id}) = \Lambda_5.
\end{equation}

\subsubsection{\texorpdfstring{When $5 \nmid g_{2, 4}$}{When 5 does not divide g24}} 
\label{paragraph0} 
By \eqref{generalnotation*} and by \eqref{eZxnx45bis} there is at most one $\sigma \in \mathfrak S $ such that $5 \mid p_{1, \sigma}(\tilde{\bf t})$ and $5 \mid p_{2, \sigma}(\tilde{\bf t})$ since $Z_{\tilde{\bf t}, \text{top}}(5) \in \{0, 1\}$. We denote by $\sigma_0$ this potential element (necessarily different from ${\rm id}$) of $\mathfrak S$. The main covering \eqref{maincovering} becomes either
\begin{equation}
\label{3203**}
\Z^2 = \Lambda(\sigma_0) \cup \bigcup_{\sigma \in \mathfrak{S} - \{\sigma_0\}}\tilde{\Lambda}^{(5)}(\sigma)
\end{equation}
or
\begin{equation}
\label{3207**}
\Z^2 = \bigcup_{\sigma \in \mathfrak{S}} \tilde{\Lambda}^{(5)}(\sigma)
\end{equation}
depending on the existence of $\sigma_0$. 

Consider the first covering \eqref{3203**}: the first lattice on the RHS has an index divisible by $5$ by \eqref{5dividesindex} . We apply Lemma \ref{lZxnx}, equation \eqref{eZxnx45bis}: with the notations there, we have $Z_{\tilde{\bf t}, {\rm top}}(5) = 1$ and $n_{\tilde{\bf t}, {\rm top}}(5) \leq 2$. Then by Remark \ref{equalityoflattices}, the last union contains at most two distinct lattices and their index is divisible by $5$. The covering \eqref{3203**} produces a contradiction, since $\Z^2$ is not equal to the union of three lattices with index divisible by $5$, see Lemma \ref{sumofdensities}.

We follow the same strategy to deduce that the decomposition \eqref{3207**} is also impossible. Since $\sigma_0$ does not exist, we have $Z_{\tilde{\bf t}, {\rm top}}(5) = 0$, so $n_{\tilde{\bf t}, {\rm top }}(5) \leq 3$. Once more we obtain a covering of $\Z^2$ by three lattices with index divisible by $5$.

\subsubsection{\texorpdfstring{When $5 \mid g_{2, 4}$}{When 5 divides g24}} 
With this condition, the coefficient of $x_2$ in \eqref{1283bis} is $0 \bmod 5$. Therefore the equations defining the $5$--enveloping lattice \eqref{954} are simplified as
$$
\tilde{\Lambda}^{(5)}(\sigma) : p_{1, \sigma}(\tilde{\bf t}) x_1 \equiv 0 \bmod 5 \quad (\sigma \in \mathfrak S).
$$ 
However, by \eqref{propgij}, we certainly have 
$$
\gcd(t_1, t_3, 5) = 1. 
$$
By Lemma \ref{lTripleVanishing}, there exists some $\mathfrak S^\flat_{\leq 2}$ (see definition \eqref{defS<k}) such that for every $\sigma \not\in \mathfrak S^\flat_{\leq 2}$, one has $p_{1, \sigma}(\tilde {\bf t}) \not\equiv 0 \bmod 5$. For those $\sigma$ we obviously have $\tilde{\Lambda}^{(5)}(\sigma ) \subseteq \Lambda_5$. After these observations, the main covering \eqref{maincovering} gives the equality
$$
\Z^2 = \Lambda^? \cup \Lambda^? \cup \Lambda_5,
$$ 
which is nonsense by Lemma \ref{length3}. The proof of Step \ref{1<b<a<5} is now complete.

\subsection{\texorpdfstring{The case $a = 3$ is impossible}{The case a = 3 is impossible}} 
\label{a=3impossible}
We continue our study of the possible values of $a$ by improving Step \ref{1<b<a<5} into

\begin{step}
\label{1<b<a<5,aneq3} 
With the notations and hypotheses of Step \ref{1<b<a<6}, one has the relations
$$
1 \leq b \leq a \leq 4 \quad \textup{ and } \quad a \neq 3.
$$
\end{step}

The proof of this step has several similarities with the proof of Step \ref{1<b<a<5}, in particular it works by contradiction. So we suppose that $a = 3$. By the equality 
\begin{equation}
\label{3/b}
\frac{D^2}{\nu} = \frac{a}{b} = \frac{3}{b}
\end{equation}
we know that $3$ divides $D$ and that
\begin{equation}
\label{3dividesindex*}
3 \mid [\Z^2: \Lambda(\sigma)] \quad \text{ for all } \sigma \in \mathfrak S
\end{equation}
by Remark \ref{useful}. We have $3 \mid d_2 D/g_{2, 4}$, so we are allowed to consider the $3$--enveloping lattices $\tilde{\Lambda}^{(3)}(\sigma)$. We have the inclusion
$$
\Lambda(\sigma) \subseteq \tilde{\Lambda}^{(3)}(\sigma) \quad (\sigma \in \mathfrak S^\flat)
$$
with
\begin{equation}
\label{1283}
\tilde{\Lambda}^{(3)}(\sigma) : p_{1, \sigma}(\tilde{\bf t}) x_1 + g_{2, 4} p_{2, \sigma}(\tilde{\bf t}) x_2 \equiv 0 \bmod 3 \quad (\sigma \in \mathfrak S^\flat)
\end{equation}
and
$$
\tilde{\Lambda}^{(3)}({\rm id}) = \Lambda_3.
$$

\subsubsection{\texorpdfstring{When $3 \nmid g_{2, 4}$}{When 3 does not divide g24}}
We appeal to Lemma \ref{lZxnx} with $x = 3$.

\vskip .2cm
$\bullet$ If $Z_{\tilde{\mathbf{t}}, \text{top}}(3) + n_{\tilde{\mathbf{t}}, \text{top}}(3) \leq 3$ and $ Z_{\tilde{\mathbf{t}}, \text{top}}(3) \leq 1$. In this situation we follow the argument presented for $a = 5$, as it was done in \S \ref{paragraph0}, leading to the impossibility of the covering \eqref{maincovering}. Indeed, we would obtain a covering of $\Z^2$ by three lattices with index $\geq 3$ by \eqref{3dividesindex*}. This is impossible by Lemma \ref{sumofdensities}.

\vskip .2cm
$\bullet$ If $Z_{\tilde{\mathbf{t}}, \text{top}}(3) = 5$. This subcase has no analogue in the case where $a = 5$. Under this assumption, we know that 
\[
\tilde{\mathbf{t}} \bmod 3 \in \mathcal T_3 := \{(0, 1, 0, 1), (0, 2, 0, 2), (1, 1, 1, 1), (2, 2, 2, 2), (1, 2, 1, 2), (2, 1, 2, 1)\}
\]
thanks to \eqref{eBadt} from Lemma \ref{lZxnx}. We have $t_1\widetilde{t_4} - \widetilde{t_2}t_3 \equiv 0 \bmod 3$. This implies that $3 \mid \frac{d_2}{g_{2, 4}}$ and therefore 
$$
9 \mid \frac{d_2D}{g_{2, 4}}.
$$
This divisibility property allows us to work with the $9$--enveloping lattices $\Lambda(\sigma) \subseteq \tilde{\Lambda}^{(9)}(\sigma)$. The $9$--enveloping lattices are given by the equations
$$
\tilde{\Lambda}^{(9)}(\sigma) : \widetilde{p_{1, \sigma}}(\tilde{\mathbf t}) x_1+ g_{2, 4} \widetilde {p_{2, \sigma}}(\tilde{\bf t}) x_2 \equiv 0 \bmod 3 \quad (\sigma \in \mathfrak S),
$$
where the notations are introduced before Lemma \ref{messfor9}. The inequalities contained in that lemma play the same r\^ole as the inequalities given by the first part of \eqref{3335}. The rest of the proof is the same: the covering \eqref{maincovering} becomes a covering of $\Z^2$ by three lattices of index divisible by $3$.
This is impossible due to Lemma \ref{sumofdensities}.

\subsubsection{\texorpdfstring{When $3 \mid g_{2, 4}$}{When 3 divides g24}}
\label{5.2.2}
By \eqref{propgij} we have $3 \nmid g_{1, 3}$. In this case, the $3$--enveloping lattices $\tilde{\Lambda}^{(3)}(\sigma)$ from \eqref{1283} become
\begin{equation}
\label{1452}
\tilde{\Lambda}^{(3)}(\sigma) : p_{1, \sigma}(\tilde{\bf t}) x_1 \equiv 0 \bmod 3 \quad (\sigma \in \mathfrak S^\flat)
\end{equation}
and $\tilde{\Lambda}^{(3)}({\rm id}) = \Lambda_3$. We apply the last part of Lemma \ref{lTripleVanishing}. We will distinguish cases based on 
\begin{align}
\label{eBad3}
|\{\sigma \in \mathfrak{S}^\flat : p_{1, \sigma}(\tilde{\bf t}) \equiv 0 \bmod 3\}| \in \{0, 1, 2, 5\}.
\end{align}
Suppose that this cardinality is at most $2$. Returning to \eqref{1452}, we see that, for all other $\sigma$ outside this set, one has $\tilde{\Lambda}^{(3)}(\sigma) \subseteq \Lambda_3$. With these properties, we deduce that the main covering \eqref{maincovering} leads to
\begin{equation}
\label{1495}
\Z^2 = \Lambda^? \cup \Lambda^? \cup \Lambda_3,
\end{equation}
which is impossible since the two $\Lambda^?$ have an index $\geq 3$ by \eqref{3dividesindex*}.

Now suppose that the cardinality in equation (\ref{eBad3}) is $5$. Then we have $\mathbf{t} \in \mathcal{T}_3^\ast$ by the last part of Lemma \ref{lTripleVanishing}, so $3$ divides $p_{1, \sigma}(\tilde{\bf t})$, $p_{2, \sigma}(\tilde{\bf t})$ but also $3$ divides $d_2/g_{2, 4}$. Thus we may consider the $9$--enveloping lattices
$$
\tilde{\Lambda}^{(9)}(\sigma) : p_{1, \sigma}(\tilde{\bf t}) x_1 + g_{2, 4} p_{2, \sigma}(\tilde{\bf t}) x_2 \equiv 0 \bmod 9.
$$
Because $\mathbf{t} \in \mathcal{T}_3^\ast$, we know that $3$ divides $p_{2, \sigma}(\tilde{\bf t})$. Since $3$ also divides $g_{2, 4}$, the above equation is equivalent to
$$
\tilde{\Lambda}^{(9)}(\sigma) : p_{1, \sigma}(\tilde{\bf t}) x_1 \equiv 0 \bmod 9.
$$
We apply the first part of Lemma \ref{lTripleVanishing} with $p = 9$. Following the same argument as above leads us to the non-existent cover (\ref{1495}).

The proof of Step \ref{1<b<a<5,aneq3} is complete.

\subsection{\texorpdfstring{The case $a = 4$ is impossible}{The case a = 4 is impossible}} 
We continue our study of the possible values of $a$ by improving Step \ref{1<b<a<5,aneq3} into

\begin{step}
\label{1<b<a<3*} 
With the notations and hypotheses of Step \ref{1<b<a<6}, one has the relations
$$
1\leq b \leq a \leq 2.
$$
In particular, the ratio $D^2/\nu$ can only be equal to $1$ or to $2$.
\end{step}

We still need to analyze $a = 4$ to arrive at a contradiction. So we suppose that $a=4$. As a consequence of the equality $D^2/\nu = 4/b$, with $b$ odd, we obtain
$$
4 \mid [\Z^2: \Lambda(\sigma)] \quad \text{ for all } \sigma \in \mathfrak S
$$
by Remark \ref{useful}. Referring to Lemma \ref{length5}, we deduce that the main covering \eqref{maincovering} must correspond to the minimal covering ${\mathcal C}_{4^6}$, so all lattices have index $4$. In particular, we have
\begin{equation}
\label{1622}
D = 4 \quad \text{ and } \quad 2^2 \Vert \nu.
\end{equation}
Since we have $4 \mid d_2D/g_{2, 4}$, we consider the $4$--enveloping lattice $\Lambda(\sigma) \subseteq \tilde{\Lambda}^{(4)}(\sigma)$ (see \eqref{954}) defined by
\begin{equation}
\label{1519}
\tilde{\Lambda}^{(4)}(\sigma) : p_{1, \sigma}(\tilde{\bf t}) x_1 + g_{2, 4}\,p_{2, \sigma}(\tilde{\bf t}) x_2 \equiv 0 \bmod 4 \quad (\sigma \in \mathfrak S^\flat)
\end{equation}
and
$$
\tilde{\Lambda}^{(4)} ({\rm id}) = \Lambda_4.
$$
We subdivide our argument according to the parity of $g_{2, 4}$.

\subsubsection{\texorpdfstring{When $2 \nmid g_{2, 4}$}{When 2 does not divide g24}}
We apply inequality \eqref{2795} of Lemma \ref{lZxnx}, which leads to two cases.

If $Z_{\tilde{\bf{t}}, \text{top}}(4) = 0$, then among the six lattices $\tilde{\Lambda}^{(4)}(\sigma)$, at most four are distinct, and they all have an index equal to $4$. So the main covering gives a covering of $\Z^2$ by at most four lattices with index $4$. This is absurd.

If $Z_{\tilde{\bf{t}}, \text{top}}(4) = 1$, write $\sigma_0$ for the corresponding automorphism. Here also, we arrive at a covering of $\Z^2$ by four lattices of index $4$, one of these being $\Lambda(\sigma_0)$, again a contradiction.

\subsubsection{\texorpdfstring{When $2 \mid g_{2, 4}$}{When 2 divides g24}}
In this case, the equation defining $\tilde{\Lambda}^{(4)}(\sigma)$ degenerates. We rather work with the second equation of \eqref{generalnotation}. Since $2\mid g_{2, 4} \mid d_2$ and since $v_2(D) = v_2(\nu) = 2$ by \eqref{1622}, we have the following inclusion
$$
\Lambda(\sigma) \subseteq \tilde{M}^{(2)}(\sigma) \quad (\sigma \in \mathfrak S^\flat),
$$
where $\tilde{M}^{(2)}(\sigma)$ is the lattice defined by
$$
\tilde{M}^{(2)}(\sigma) : P_{1, \sigma}(t_1, t_3) x_1 \equiv 0 \bmod 2.
$$
Since $(t_1, t_2, t_3, t_4)$ are coprime, we necessarily have $(t_1, t_3) \not \equiv (0, 0) \bmod 2$. By the explicit definition of the polynomials $P_{1, \sigma}$, we deduce
\begin{itemize}
\item when $(t_1, t_3) \equiv (1, 0) \bmod 2$, we have $\tilde M^{(2)}(\sigma) = \Lambda _2$, for the four cases 
$$
\sigma \in \{T_2^{-1} R T_2, \, T_2^{-1} R^2 T_2, \,T_2^{-1} S T_2, \, T_2^{-1} RS T_2\}.
$$ 
Since $\Lambda({\rm id}) \subseteq \Lambda_2$, the main covering \eqref{maincovering} gives the equality
$$
\Z^2 = \Lambda_2 \cup \Lambda^?,
$$
which is impossible.
\item when $(t_1, t_3) \equiv (0, 1) \bmod 2$, we also have the equality $\tilde{M}^{(2)}(\sigma) = \Lambda_2$ for four $\sigma$ and we arrive at a contradiction.
\item when $(t_1, t_3) \equiv (1, 1) \bmod 2$, the equality $\tilde M^{(2)}(\sigma) = \Lambda_2$ holds for four $\sigma$, and here also we reach a dead end.
\end{itemize}
The proof of Step \ref{1<b<a<3*} is complete.

\section{The dual covering} 
The strong equality $D^2/\nu \in \{1, 2\}$ stated in Step \ref{1<b<a<3*} however gives no detail on $D$ and $\nu$ separately. In the above section we exploited the properties of the main covering \eqref{maincovering} to gain information on $a$ and $b$. In order to understand the parameters $D$ and $\nu$, we will incorporate the study the dual covering \eqref{dualcovering}.

\subsection{\texorpdfstring{Description of the lattices $\Gamma$ by their equations}{Description of the lattices Gamma by their equations}} 
\begin{lemma}
\label{Lemma6.1*} 
For $\sigma \in \mathfrak S$, let $\Gamma (\sigma)$ be the lattice defined in \eqref{defGammabis}. Then we have the equations
$$
\Gamma(\textup{id}): 
\begin{cases}
D x_1 &\equiv 0 \bmod 1 \\
D x_2 &\equiv 0 \bmod \nu,
\end{cases}
$$
 
$$
\Gamma(T_2^{-1}RT_2):
\begin{cases}
D \nu (t_1t_2 + t_2t_3 + t_3t_4) x_1 + D (t_2^2 + t_2t_4 + t_4^2) x_2 &\equiv 0 \bmod d_2 \nu \\
D \nu (t_1^2 + t_1t_3 + t_3^2) x_1 + D (t_1t_2 + t_1t_4 + t_3t_4) x_2 &\equiv 0 \bmod d_2 \nu,
\end{cases}
$$
 
$$
\Gamma(T_2^{-1}R^2T_2):
\begin{cases}
D \nu (t_1t_2 + t_1t_4 + t_3t_4) x_1 + D (t_2^2 + t_2t_4 + t_4^2) x_2 &\equiv 0 \bmod d_2 \nu \\
D \nu (t_1^2 + t_1t_3 + t_3^2) x_1 + D (t_1t_2 + t_2t_3 + t_3t_4) x_2 &\equiv 0 \bmod d_2 \nu,
\end{cases}
$$

$$
\Gamma(T_2^{-1}ST_2):
\begin{cases}
D \nu (t_3t_4 - t_1t_2) x_1 + D (t_4^2 - t_2^2) x_2 &\equiv 0 \bmod d_2 \nu \\
D \nu (t_1^2 - t_3^2) x_1 + D (t_1t_2 - t_3t_4) x_2 &\equiv 0 \bmod d_2 \nu,
\end{cases}
$$

$$
\Gamma(T_2^{-1}RST_2):
\begin{cases}
D \nu (t_1t_2 + t_1t_4 + t_2t_3) x_1 + D (t_2^2 + 2t_2t_4) x_2 &\equiv 0 \bmod d_2 \nu \\
D \nu (t_1^2 + 2t_1t_3) x_1 + D (t_1t_2 + t_1t_4 + t_2t_3) x_2 &\equiv 0 \bmod d_2 \nu,
\end{cases}
$$
and
$$
\Gamma(T_2^{-1}R^2ST_2):
\begin{cases}
D \nu (t_1t_4 + t_2t_3 + t_3t_4) x_1 + D (t_4^2 + 2t_2t_4) x_2 &\equiv 0 \bmod d_2 \nu \\
D \nu (t_3^2 + 2t_1t_3) x_1 + D (t_1t_4 + t_2t_3 + t_3t_4) x_2 &\equiv 0 \bmod d_2 \nu.
\end{cases}
$$
\end{lemma}

\begin{proof}
One can prove this by a direct computation. Instead we may also appeal to Lemma \ref{lLatticeEquationsbis} to exploit the obvious link between $\Lambda(\sigma)$ and $\Gamma(\sigma)$: to obtain $\Gamma (\sigma)$ we postmultiply by $\gamma$ but to obtain $\Lambda(\sigma)$ we premultiply by $\gamma^{-1}$ (see \eqref{defLambdabis} and \eqref{defGammabis}). But the matrices $\gamma$ and $\gamma^{-1}$ are diagonal and this property explains that the variation in these two calculations reduces to putting $D$ and $\nu$ in different places without perturbating the quadratic forms $p_{i,\sigma}({\bf t})$ and $P_{i, \sigma}({\bf t})$ in $\mathbf t$ (see \eqref{generalnotation}).
\end{proof}

\subsection{\texorpdfstring{Key claim: finitely many possibilities for $D$}{Key claim: finitely many possibilities for D}}
We now translate Step \ref{1<b<a<3*} in terms of $(D, \nu)$.

\begin{step}
With the notations and hypotheses of Step \ref{1<b<a<6} one has
\begin{equation}
\label{afirstlist*}
(D, \nu) \in \{(2, 4),\, (3,9),\, (4, 16),\, (5, 25),\, (2, 2),\, (4, 8),\, (6, 18), \, (8, 32), \, (10, 50)\}.
\end{equation}
\end{step}

\begin{proof}
Since $D \mid \nu$, we write $\nu =D \mu$ for some integer $\mu \geq 1$. Lemma \ref{atmostfive} implies the inequality $1\leq \mu \leq 5$. Since $D^2/\nu = D/\mu= a/b$, Step \ref{1<b<a<3*} leads to the equality $(D, \nu) = (\mu, \mu^2)$ or $(D, \nu) = (2\mu, 2\mu^2)$. We recover the list given in \eqref{afirstlist*}, since the pair $(D, \nu) = (1, 1)$ is not allowed by \eqref{conditionsforDandnu}.
\end{proof}

We continue to shrink the set of possible values of the pair $(D, \nu)$ by forbidding the prime $p=5$ as a possible divisor of $D$. We have 

\begin{step}
\label{1696bis}
With the notations and hypotheses of Step \ref{1<b<a<6} one has
\begin{equation}
\label{asecondlist**}
(D, \nu) \in \{(2, 4), \, (3,9), \, (4, 16), \, (2, 2), \, (4, 8), \, (6, 18), \, (8, 32)\}.
\end{equation}
\end{step}

\begin{proof}
For the sake of contradiction suppose that $5 \mid D$. By \eqref{798} we still have $[\Z^2 : \Lambda(\sigma)] \geq 5$. The proof is similar to the proof of Step \ref{1<b<a<5} (see \S \ref{a=5isimpossible}) and we use the same notations. So we appeal to the $5$--enveloping lattices based on the relations \eqref{1300}, \eqref{1283bis} and \eqref{1308}. We mimic the proof of Step \ref{1<b<a<5} based on a discussion of the value of $\gcd(5, g_{2, 4})$ and on Lemmas \ref{lZxnx} and \ref{lTripleVanishing}. We then prove that the main covering \eqref{maincovering} is impossible. So $D$ can not be divisible by $5$. Eliminating the corresponding $D$ from the list \eqref{afirstlist*}, we find the list \eqref{asecondlist**} and Step \ref{1696bis} is proved. 
\end{proof}

We continue to shorten the list \eqref{asecondlist**} by proving

\begin{step}
With the notations and hypotheses of Step \ref{1<b<a<6} one has
\[
(D, \nu) \in \{(2, 4), \, (3,9), \, (4, 16), \, (2, 2), \, (4, 8), \, (6, 18)\}.
\]
\end{step}

\begin{proof} 
We suppose that $D = 8$ and $\nu = 32$ to arrive at a contradiction. By \eqref{798}, we have $[\Z^2 : \Lambda(\sigma)] \geq 4$ for $\sigma \in \mathfrak{S}$. As a a direct consequence of \eqref{Idbis} we also have $[\Z^2: \Lambda({\rm id})] = 8$. The main covering \eqref{maincovering} is then built with one component with index $8$ and five components with index at least $4$. From such a covering, it is impossible to extract a minimal covering as described in Lemmas \ref{length4}, \ref{length5} and \ref{length6}. So the pair $(D, \nu) = (8, 32)$ is not allowed.
\end{proof}

Another progress is the following one

\begin{step}
\label{1672*}
With the notations and hypotheses of Step \ref{1<b<a<6} one has
\begin{equation}
\label{asecondlistbis}
(D, \nu) \in \{(2, 4), \, (4, 16), \, (2, 2), \, (4, 8)\}.
\end{equation}
\end{step}

\begin{proof} 
We suppose that $(D, \nu) \in \{(3, 9), (6, 18)\}$. In both cases we have $3 = \nu/D$. By \eqref{798} we have the lower bound
\begin{equation}
\label{index>3}
[\Z^2 : \Lambda(\sigma)] \geq 3.
\end{equation} 
In \S \ref{a=3impossible}, for the proof of Step \ref{1<b<a<5,aneq3} the hypothesis $a = 3$ led to the property \eqref{3dividesindex*}, which implies \eqref{index>3}. The proof of Step \ref{1<b<a<5,aneq3} is readily adapted to this slight generalization.
\end{proof}

\subsection{\texorpdfstring{Restricting the prime divisors of $d_2$}{Restricting the prime divisors of d2}} 
The index of the lattices $\Lambda(\sigma)$ and $\Gamma (\sigma)$ obviously depends on the value of $d_2$ (see Lemmas \ref{lLatticeEquationsbis} and \ref{Lemma6.1*}). We now restrict the possible values of $d_2$ by the following step.

\begin{step}
\label{1762bis} 
With the above hypotheses and notations, we have the following restrictions for the triples $(d_2, D, \nu)$:
\begin{equation}\label{1693}
\begin{cases}
(D, \nu) \in \{(2, 4), \, (4, 16),\, \, (2, 2),\, (4, 8)\}, \\ p\mid d_2 \Rightarrow p \in \{2, 3\}.
\end{cases}
\end{equation}
\end{step}

\begin{proof}
For the sake of contradiction, suppose that there exists some prime $p \mid d_2$ with $p \geq 5$. Recall that $g_2(\sigma) = d_2/d_2(\sigma)$ (see \S \,\ref{indicator}). If $\sigma \neq \sigma'$ do not both have order $3$, then we have 
\begin{equation}
\label{3768bis}
\gcd(g_2(\sigma), g_2(\sigma')) \in \{1, 3\}
\end{equation}
by Lemma \ref{lSX3.2}. We remark that $p \nmid g_2(\sigma) \Rightarrow p \mid d_2(\sigma)$. By this remark and by \eqref{3768bis}, we are allowed to split our proof in two cases:
\begin{itemize}
\item {\bf Case 1:} If $p$ divides at most one $g_2(\sigma)$ ($\sigma \in \mathfrak S^\flat$), then $p$ does not divide at least four $g_2(\sigma)$, so $p \mid d_2(\sigma)$ for at least four different $\sigma \in \mathfrak{S}^\flat$, 
\item {\bf Case 2:} If $p$ divides at least two $g_2(\sigma)$ ($\sigma \in \mathfrak S^\flat$), then the number of these $\sigma$ is exactly two and they both have order three. In particular, $p \mid d_2(\sigma)$ for $\sigma$ in the set $\{T_2^{-1}ST_2, T_2^{-1}RST_2, T_2^{-1}R^2ST_2\}$, which are precisely the elements of order two.
\end{itemize}

\subsubsection{Case 1} 
We use the functions $d_2(\sigma) $ and $g_2(\sigma)$ to write in another way the equations of the lattices $\Lambda(\sigma)$ for $\sigma \in \mathfrak S^\flat$ when $p \mid d_2(\sigma)$. Consider the equation of such a lattice $\Lambda(\sigma)$ schematically written as in Lemma \ref{lLatticeEquationsbis} (see \S \ref{indicator})
$$ 
\Lambda(\sigma):
\begin{cases}
a' x_1 + b' x_2 & \equiv 0 \bmod d_2 D\\
\nu (c'x_1 +d' x_2) & \equiv 0 \bmod d_2 D.
\end{cases}
$$
Dividing by $g_2(\sigma)$ and using the property \eqref{g2=gcd}, we obtain, for certain coprime integers $(a, b, c, d)$, the equivalent system of equations
\begin{equation}
\label{3788}
\Lambda(\sigma): 
\begin{cases} 
ax_1 + b x_2 & \equiv 0 \bmod d_2(\sigma) D\\
\nu (c x_1 + d x_2) & \equiv 0 \bmod d_2(\sigma) D.
\end{cases}
\end{equation}
But $p$ divides $d_2(\sigma)D$ and it is coprime with $\gcd (a, b, \nu c, \nu d)$ as a consequence of \eqref{asecondlistbis}. This implies that $\Lambda(\sigma)$ ($\sigma \in \mathfrak S^\flat$) has an index divisible by $p$. When $p \geq 7$, the covering \eqref{maincovering} is nonsense: four of the six lattices would have an index divisible by $p\geq 7$ and we do not encounter an associated minimal covering in Lemmas \ref{length3}, \ref{length4}, \ref{length5} and \ref{length6}. In the case $p = 5$ we also have to avoid the minimal covering $\mathcal C_{5^6}$ of Lemma \ref{length6}, where the indices of all the lattices are equal to $5$. But in our case, we have the equality (see Remark \ref{rLidSimple})
\begin{equation}
\label{indexofidbis}
[\Z^2: \Lambda({\rm id})] = D,
\end{equation}
which certainly does not divide $5$ by \eqref{asecondlistbis}. 

\subsubsection{Case 2} 
We make the same considerations as in Case 1. Since there are three $\sigma$ with order $2$, the equations \eqref{3788} imply that in the covering \eqref{maincovering} at least three lattices have an index divisible by $p\geq 5$. The lattice $\Lambda({\rm id})$ has an index which is divisible by no prime $\geq 5$, by \eqref{indexofidbis} and by Step \ref{1672*}. By considering Lemmas \ref{length3}, \ref{length4}, \ref{length5} and \ref{length6}, we deduce that in the main covering \eqref{maincovering} the three lattices associated with $\sigma$ of order 2 are superfluous, which means that the extracted minimal covering has necessarily length three. So this covering leads to the equality
\[
\bigcup_{\sigma \in \{\text{id}, T_2^{-1}RT_2, T_2^{-1}R^2T_2\}} \Lambda(\sigma) = \Z^2.
\]
But this equality contradicts \eqref{backtoC3}. The proof of Step \ref{1762bis} is complete.
\end{proof}

\subsection{\texorpdfstring{Further restrictions on $d_2 D$, part 1}{Further restrictions on d2 D, part 1}}
\label{part3}
We continue to restrict the possible values of the triples $(d_2, D, \nu)$ by eliminating the value $D = 4$ in the set \eqref{1693}.

\begin{step}
\label{Dneq4} 
With the above hypotheses and notations, we have the following restrictions for the triples $(d_2, D, \nu)$:
\begin{equation}
\label{1752}
\begin{cases}
(D, \nu) \in \{(2, 4), \, (2, 2)\}, \\ p\mid d_2 \Rightarrow p \in \{2, 3\}.
\end{cases}
\end{equation}
\end{step}

For the sake of contradiction, we suppose that $D = 4$. We recall the notations \eqref{defgij}, \eqref{propgij} and \eqref{defttilde} and we notice 
$$
4 \mid d_2D/g_{2, 4}.
$$
Thus we are allowed to use the $4$--enveloping lattices $\Lambda(\sigma) \subseteq \tilde{\Lambda}^{(4)}(\sigma)$ for $\sigma \in \mathfrak S^\flat$ (see \eqref{954}, \eqref{958}, \eqref{962} and \eqref{1519}). The main covering \eqref{maincovering} implies the following covering
\begin{equation}
\label{1754*}
\Z^2 = \Lambda_4 \cup \bigcup_{\sigma \in \mathfrak S^\flat} \tilde{\Lambda}^{(4)}(\sigma), 
\end{equation}
since $\Lambda({\rm id}) = \Lambda_D = \Lambda_4$. We will use the following facts from Lemma \ref{lZxnx}, with $x = 4$, to shorten the number of lattices participating in the covering \eqref{1754*} in order to arrive at a contradiction. 

\begin{itemize}
\item[(i)] if $2 \nmid g_{2, 4}$, we apply Lemma \ref{lZxnx}. The second part of \eqref{2795} of Lemma \ref{lZxnx}, for $x = 4$, shows that there is at most one $\sigma \in \mathfrak S^\flat$ such that $2 \mid p_{1, \sigma}(\tilde{\bf t})$ and $2 \mid p_{2, \sigma}(\tilde{\bf t})$. We denote this potential element by $\sigma_0$ and its existence corresponds to the equality $Z_{\tilde{\bf t}, {\rm top}}(4) = 1$. By Remark \ref{equalityoflattices} and by the first part of \eqref{2795}, the main covering \eqref{1754*} becomes
\begin{gather}
\label{kappa0exists**}
\text{ if }\sigma_0 \text { exists: } \Z^2 = \Lambda_4 \cup \Lambda(\sigma_0) \cup \, \bigcup_{\sigma \in \mathfrak S_{\leq 2}^\flat} \tilde{\Lambda}^{(4)}(\sigma) \\
\label{kappa0doesnot**}
\text{ if } \sigma_0 \text { does not exist: } \Z^2 = \Lambda_4 \cup \, \bigcup_{\sigma \in \mathfrak S_{\leq 3}^\flat} \tilde{\Lambda}^{(4)}(\sigma) 
\end{gather} 
for some $\mathfrak S_{\leq 2}^\flat$ and $\mathfrak S_{\leq 3}^\flat$ (see \eqref{defS<k}). In these two coverings, there are at most 4 lattices, all different from $\Z^2$. Since in these formulas the enveloping lattices $ \tilde{\Lambda}^{(4)}(\sigma)$ and $\Lambda_4$ have index $4$, the coverings \eqref{kappa0exists**} and \eqref{kappa0doesnot**} are nonexistent, since it is impossible to extract a minimal covering appearing in Lemma \ref{length3} or \ref{length4}.


\item[(ii)] if $2 \mid g_{2, 4}$, then the coefficient of $x_2$ in the equation \eqref{1519} is always even. We continue to denote by $\sigma_0$ the potential $\sigma$ such that $2 \mid p_{1, \sigma}(\tilde{\bf t})$ and $2 \mid p_{2, \sigma}(\tilde{\bf t})$. We investigate different situations

\begin{itemize}
\item If $\sigma_0$ exists, we will use the fact that the index of $\Lambda_4$ is equal to $D = 4$ and that the index of $\Lambda(\sigma_0)$ is at least $2$. For $\sigma \in \mathfrak S^\flat - \{\sigma_0\}$, we appeal to Lemma \ref{lZxnx}, formula \eqref{2795}, alinea $(iii)$ to conclude that the pair of coefficients $ (p_{1, \sigma}(\tilde{\bf t}), p_{2, \sigma}(\tilde{\bf t}))$ modulo $4$ has its values in the set
$$
(p_{1, \sigma}(\tilde{\bf t}), p_{2, \sigma}(\tilde{\bf t})) \in \{(1, 0), (1, 1), (1, 2), (1, 3), (3, 0), (3, 1), (3, 2), (3, 3)\},
$$
where we notice that the first component of each pair is odd. Returning to \eqref{1519}, we see that every $\tilde{\Lambda}^{(4)}(\sigma)$ ($ \sigma \in \mathfrak S^\flat - \{\sigma_0\}$) is included in the lattice $\Lambda_2$, since $g_{2, 4}$ is even. Hence the covering \eqref{kappa0exists**} leads to a contradiction, since it reduces to
\begin{equation}
\label{10000}
\Z^2 = \Lambda_2 \cup \Lambda^?,
\end{equation}
which is not compatible with Lemma \ref{length2}. 

\item If $\sigma_0$ does not exist, we use Lemma \ref{lZxnx}, formula \eqref{2795}, alinea $(ii)$. We suppose that $(0, 1)$ (and equivalently $(0, 3)$) does not belong to the set 
$$
\left\{\left(\alpha p_{1, \sigma }(\tilde{\mathbf t}), \alpha p_{2, \sigma}(\tilde{\mathbf t})\right) : \alpha \in (\Z /4\Z)^\ast, \sigma \in S_{\tilde{\mathbf t}, {\rm top}}(4)\right\}.
$$
We will treat the case where $(2, 1)$ replaces $(0, 1)$ below. We know that the pair of coefficients $(p_{1, \sigma}(\tilde{\bf t}), p_{2, \sigma}(\tilde{\bf t}))$ ($\sigma \in \mathfrak S^\flat$) modulo $4$ has its values in the set 
$$
(p_{1, \sigma}(\tilde{\bf t}), p_{2, \sigma}(\tilde{\bf t})) \in \{(1, 0), (1, 1), (1, 2), (1, 3), (2, 1), (2, 3), (3, 0), (3, 1), (3, 2), (3, 3)\}.
$$
We return to the equation \eqref{1519} to deduce, by straightforward calculations, that in the present situation we have (recall that $g_{2, 4} $ is even) the inclusions 
$$
\tilde{\Lambda}^{(4)}(\sigma) \subseteq
\begin{cases}
\Lambda_2 &\text{ if } 2 \nmid p_{1, \sigma}(\tilde{\bf t}) \\
\{(x_1, x_2) : x_1 + (g_{2, 4}/2) x_2 \equiv 0 \bmod 2\} & \text{ if } 2\mid p_{1, \sigma}(\tilde{\bf t}).
\end{cases}
$$
where the lattices, on the RHS, are independent from $\sigma$. We deduce that the covering \eqref{kappa0doesnot**} is simplified to
\begin{equation}
\label{20000}
\Z^2 = \Lambda_2 \cup \Lambda^?.
\end{equation}
This is impossible due to Lemma \ref{length2}. 

It remains to analyze the case
$$
(2, 1) \not \in \left\{\left(\alpha p_{1, \sigma }(\tilde{\mathbf t}), \alpha p_{2, \sigma}(\tilde{\mathbf t})\right) : \alpha \in (\Z /4\Z)^\ast, \sigma \in S_{\tilde{\mathbf t}, {\rm top}}(4)\right\}.
$$
In this case, $(p_{1, \sigma}(\tilde{\bf t}), p_{2, \sigma}(\tilde{\bf t}))$ modulo $4$ has its values in the set 
$$
(p_{1, \sigma}(\tilde{\bf t}), p_{2, \sigma}(\tilde{\bf t})) \in \{(1, 0), (1, 1), (1, 2), (1, 3), (0, 1), (0, 3), (3, 0), (3, 1), (3, 2), (3, 3)\}.
$$
If $4 \nmid g_{2, 4}$, then we consider the $4$--enveloping lattices from equation \eqref{1519}. By looking at the above set of values for $(p_{1, \sigma}(\tilde{\bf t}), p_{2, \sigma}(\tilde{\bf t}))$, we obtain the inclusions
$$
\tilde{\Lambda}^{(4)}(\sigma) \subseteq
\begin{cases}
\Lambda_2 &\text{ if } 2 \nmid p_{1, \sigma}(\tilde{\bf t}) \\
\{(x_1, x_2) : x_2 \equiv 0 \bmod 2\} & \text{ if } 2\mid p_{1, \sigma}(\tilde{\bf t})
\end{cases}
$$
for $\sigma \in \mathfrak{S}^\flat$, and we finish the argument as before.

So suppose that $4 \mid g_{2, 4}$, hence $4 \mid d_2$. We now distinguish two further cases.
\begin{itemize}
\item If $(D, \nu) = (4, 16)$, all lattices have index at least $4$ (since $\Gamma({\rm id})$ has index $4$). This situation is readily handled by looking at the top equations of the system \eqref{generalnotation} and finding a $\sigma$ with $p_{1, \sigma}(\tilde{\bf t}) \not \equiv 0 \bmod 2$, see Lemma \ref{lTripleVanishing}.
\item If $(D, \nu) = (4, 8)$, we switch to the bottom equations of \eqref{generalnotation}. After dividing by $\nu$, we obtain a congruence modulo $d_2D/\nu$, which is divisible by $2$. An explicit study of the polynomial $P_{1, \sigma}$ shows that it vanishes modulo $2$ at $\tilde{\bf t}$ for at most one $\sigma$.  For all the other $\sigma$, we have $\Lambda(\sigma) \subseteq \Lambda_2$ thus leading to an impossible covering.
\end{itemize}
\end{itemize}
\end{itemize}

\noindent The proof of Step \ref{Dneq4} is complete.

\subsection{\texorpdfstring{Further restrictions on $d_2 D$, part 2}{Further restrictions on d2 D, part 2}}
We continue to restrict the possible values of the triples $(d_2, D, \nu)$ by eliminating the even values of $d_2$ in the set \eqref{1752}.

\begin{step}
\label{d2odd} 
With the above hypotheses and notations, we have the following restrictions for the triples $(d_2, D, \nu)$:
$$
\begin{cases}
(D, \nu) \in \{(2, 4), \, (2, 2)\}, \\ 
p \mid d_2 \Rightarrow p = 3.
\end{cases}
$$
\end{step}

For the sake of contradiction, we suppose $(D, \nu) \in \{(2, 4), (2, 2)\}$, but $v_2(d_2) \geq 1$. The proof mimics the proof of Step \ref{Dneq4} but is more delicate, since the lattice $\Lambda({\rm id})$ is now $\Lambda_2$ instead of the smaller lattice $\Lambda_4$ as in Step \ref{Dneq4}. We write the covering \eqref{maincovering} as
$$
\Z^2 = (\Lambda_2 \setminus \Lambda_4) \cup \mathcal{L},
$$
with
$$
\mathcal L = \Lambda_4 \cup \bigcup_{\sigma \in \mathfrak S^\flat} \Lambda(\sigma).
$$
The set $\mathcal L$ is well suited to apply Lemma \ref{lZxnx} with $x = 4$. We explain how to modify the proof of Step \ref{Dneq4} to obtain Step \ref{d2odd}.

\subsubsection{\texorpdfstring{Case where $v_2(g_{2, 4}^2) < v_2(d_2 D)$}{Case where g24 has small valuation}}  
This inequality combined with the hypothesis $D = 2$ leads to
$$
4 \mid d_2D/g_{2, 4},
$$ 
which allows us to consider the $4$--enveloping lattices $\tilde{\Lambda}^{(4)}(\sigma)$ for $\sigma\in \mathfrak S^\flat$ (see \eqref{954}).

\begin{itemize}
\item if $2 \nmid g_{2, 4}$, the proof is the same as in \S \ref{part3}\,(ii) with the help of \eqref{2795}, alinea {\it (i)}.

\item if $2 \mid g_{2, 4}$, we follow the proof of \S \ref{part3}\,(iii). We now have $\Lambda({\rm id}) = \Lambda_2$ instead of $\Lambda_4 \subset \Lambda_2$. This alteration does not modify the final equalities \eqref{10000} ($\sigma_0$ exists) or \eqref{20000} ($\sigma_0$ does not exist). However, we need to modify the last paragraph of \S \ref{part3}\,(iii), when we handle the case $4 \mid g_{2, 4}$. In our current situation, we switch to the bottom equations (as in the case $(D, \nu) = (4, 8)$ before). An explicit study of the polynomial $P_{1, \sigma}$ shows that it vanishes modulo $2$ at $\tilde{\bf t}$ for at most one $\sigma$, and we are done.
\end{itemize}

\noindent So the hypothesis $2 \mid d_2$ is incorrect when $v_2(g_{2, 4}^2) < v_2(d_2 D)$ and $(D, \nu) \in \{(2, 2), (2, 4)\}$.

\subsubsection{\texorpdfstring{Case where $v_2(g_{2, 4}^2) \geq v_2(d_2 D)$}{Case where g24 has large valuation}}
The initial assumption $v_2(d_2D) \geq 2$ implies the inequality $v_2(g_{2, 4}) \geq 1$. Furthermore, since $g_{2, 4} \mid d_2$, we always have $v_2(g_{2, 4}) < v_2(d_2 D)$. So we have the conditions
$$
v_2(g_{2, 4}^2) \geq v_2(d_2 D) \text{ and } 1 \leq v_2(g_{2, 4}) < v_2(d_2 D). 
$$ 
Recalling the notations \eqref{defgij} and \eqref{defttilde} we prove the following  

\begin{lemma}
\label{generalsituation*} 
Let $D \geq 1$ be an integer. For $\sigma \in \mathfrak S^\flat$, suppose that the lattice $\Lambda(\sigma)$ is proper. Let ${\bf t} =(t_1, t_2, t_3, t_4) \in \Z^4$ be a 4--tuple of coprime integers such that $d_2 = \vert t_1t_4 -t_2t_3 \vert$. We assume that $v_2(g_{2, 4}^2) \geq v_2(d_2 D)$ and $1 \leq v_2(g_{2, 4}) < v_2(d_2 D)$. Then, with at most one exception that we will denote by $\sigma_\dag$, one has for every $\sigma$ in the set $\mathfrak S^\flat$ 
\begin{enumerate}
\item either $\Lambda(\sigma) \subseteq \Lambda_2$,
\item or $[\Z^2 : \Lambda(\sigma) ]$ is odd (hence $\geq 3$).
\end{enumerate}
In particular, the first case happens when $p_{1, \sigma}(\tilde{\bf t}) \not\equiv 0 \bmod 2$. 
\end{lemma}

\begin{proof} 
Let $T := \gcd(d_2D, 2^\infty)$ and let $T_1$ be the odd integer defined by $T_1 := d_2D/T$. By the Chinese Remainder Theorem, the two equations of \eqref{generalnotation} are equivalent to the system of four equations modulo $T$ and $T_1$, that, with obvious notations, we write as
\begin{equation}
\label{System*}
\Lambda(\sigma):
\begin{cases} 
g_{2, 4} p_{1, \sigma}(\tilde{\bf t}) x_1& \equiv 0 \bmod T\\
g_{2, 4} p_{1, \sigma}(\tilde{\bf t}) x_1 + g_{2, 4}^2 p_{2, \sigma}(\tilde{\bf t}) x_2 & \equiv 0 \bmod T_1\\
\nu (P_{1, \sigma}(\tilde{\bf t}) x_1 + g_{2, 4} P_{2, \sigma}(\tilde{\bf t}) x_2 ) & \equiv 0 \bmod T\\
\nu (P_{1, \sigma}(\tilde{\bf t}) x_1 + g_{2, 4} P_{2, \sigma}(\tilde{\bf t}) x_2 )& \equiv 0 \bmod T_1,
\end{cases}
\end{equation}
where we used the inequality $v_2(g_{2, 4}^2) \geq v_2(d_2 D)=v_2(T)$ to shorten the first equation. Our discussion splits in the following cases:

-- If $v_2(g_{2, 4} p_{1, \sigma}(\tilde{\bf t})) < v_2(T)$, the first equation shows that $\Lambda(\sigma) \subseteq \Lambda_2$. In particular, since we assumed that $v_2(g_{2, 4}) < v_2(d_2D)$, this case certainly happens when $2 \nmid p_{1, \sigma}(\tilde{\bf t})$.

-- If $v_2(g_{2, 4} p_{1, \sigma}(\tilde{\bf t})) \geq v_2(T)$, the first equation gives no more constraints. Suppose now that $v_2(\nu) < v_2(T)$. Then the third equation of the  system \eqref{System*} leads to the equation
$$
P_{1, \sigma}(\tilde{\bf t}) x_1 + g_{2, 4} P_{2, \sigma}(\tilde{\bf t}) x_2 \equiv 0 \bmod 2,
$$
which reduces to
$$
P_{1, \sigma}(\tilde{\bf t}) x_1 \equiv 0 \bmod 2.
$$
But the hypothesis $v_2(g_{2, 4}) \geq 1$ implies that among the integers $t_1$ and $t_3$ at least one is odd. By a direct study of the explicit values of $P_{1, \sigma}$ we have already met the equality
$$
\left\vert \left\{\sigma \in \mathfrak S^\flat : P_{1, \sigma}(\tilde{\bf t})\equiv 0 \bmod 2\right\}\right\vert = 1.
$$
Denote by $\sigma_\dag$ this element. So for $\sigma \neq \sigma_\dag$, one has again $\Lambda(\sigma) \subseteq \Lambda_2$.

-- If $v_2(g_{2, 4} p_{1, \sigma}(\tilde{\bf t})) \geq v_2(T)$ and $v_2(\nu) \geq v_2(T)$, the first and the third equations of the system \eqref{System*} disappear. We are left with the second and fourth equations, which imply that the index of $\Lambda(\sigma)$ is necessarily odd. This corresponds to case {\it 2.} of Lemma \ref{generalsituation*}. 
\end{proof}

We now prove that the covering \eqref{maincovering} leads to a contradiction by appealing to Lemma \ref{generalsituation*}. By Lemma \ref{lTripleVanishing}, we know that there exists $\mathfrak S^\flat_{\leq 2}$ such that 
$$
p_{1, \sigma}(\tilde{\bf t}) \not \equiv 0 \bmod 2 \quad \text{ for } \sigma \not \in \mathfrak S^{\flat}_{\leq 2}.
$$
For $\sigma \not\in \mathfrak S^\flat_{\leq 2}$, one has $\Lambda(\sigma) \subseteq \Lambda_2$. We will only treat the case where $\mathfrak S^{\flat}_{\leq 2}$ contains exactly  two elements  that we call $\sigma_1$ and $\sigma_2$ since it is the worst case. If it exists, $\sigma_\dag$ belongs to the set $\{ \sigma_1,\, \sigma_2\}$.

Since $D = 2$, we have $\Lambda({\rm id}) = \Lambda_2$. Therefore Lemma \ref{generalsituation*} transforms the main covering \eqref{maincovering} into 
\begin{equation}
\label{1859}
\Z^2 =\Lambda_2 \cup \Lambda(\sigma_1) \cup \Lambda(\sigma_2). 
\end{equation}
At most for one $\sigma_i$, we merely know that $\Lambda(\sigma_i)$ is proper, this corresponds to $\sigma_i =\sigma_\dag$ and $\Lambda(\sigma_\dag) =\Lambda^?$. For the remaining $\Lambda(\sigma_j)$ ($j\neq i$), we know that it is either contained in $\Lambda_2\ $ or it has an odd index $\geq 3$. In any of these cases, the equality \eqref{1859} is absurd since we find no corresponding minimal covering in Lemma \ref{length3}.
 
We considered all the possible cases related to the hypothesis $2 \mid d_2$. They all lead to a dead end. The proof of Step \ref{d2odd} is complete.

\subsection{\texorpdfstring{Further restrictions on $d_2 D$, part 3}{Further restrictions on d2 D, part 3}} 
We continue to restrict the possible values of the triples $(d_2, D, \nu)$ by eliminating the triples such that $9 \mid d_2 $. In particular, this implies that the set of such triples is finite. We will prove 

\begin{step}
\label{9doesnotdividebis} 
With the above hypotheses and notations, we have the following restrictions for the triples $(d_2, D, \nu)$:
$$
\begin{cases}
(D, \nu) \in \{(2, 2),\, (2, 4)\}, \\
d_2 \in \{1, 3\}.
\end{cases}
$$
\end{step}

The proof mimics the proof of some steps above with necessary modifications. So we have $D = 2$ and for the sake of contradiction, we suppose that $9 \mid d_2$. We recall the notations and properties \eqref{defgij}, \eqref{propgij} and \eqref{defttilde}. We also recall the $\ell$--enveloping lattice $\tilde{\Lambda}^{(\ell)}(\sigma)$ defined in \eqref{954}, when $\ell \mid d_2D/g_{2, 4}$ and $\sigma \in \mathfrak S^\flat$. In this subsection, we will work with $\ell = 3$ or $\ell = 9$. We will distinguish cases according to the value of $v_3(g_{2, 4})$.

\subsubsection{\texorpdfstring{Case where $v_3(g_{2, 4}) = 0$}{Case where g24 has zero 3-adic valuation}} 
We apply Lemma \ref{lZxnx} ($x = 3$) and split our discussion according to the value of $Z_{\tilde{\mathbf{t}}, \text{top}}(3) \in \{0, 1, 5\}$. 
\begin{itemize}
\item If $Z_{\tilde{\mathbf{t}}, \text{top}}(3) = 0$. We have the equalities of coverings
\[
\Z^2 = \bigcup_{\sigma \in \mathfrak S} \Lambda(\sigma) = \Lambda_2 \cup \bigcup_{\sigma \in \mathfrak{S}_{\leq 3}^\flat } \tilde{\Lambda}^{(3)}(\sigma)
\]
for some $\mathfrak{S}_{\leq 3}^\flat$. This covering is inadmissible, since the last three lattices have their index equal to $3$ and the first lattice has index divisible by $2$. This corresponds to no minimal covering in Lemmas \ref{length3} and \ref{length4}.

\item If $Z_{\tilde{\mathbf{t}}, \text{top}}(3) = 1$. The main covering \eqref{maincovering} is written as 
\begin{equation}
\label{2091}
\Z^2 = \bigcup_{\sigma \in \mathfrak S} \Lambda(\sigma) = \Lambda_2\cup \Lambda^? \cup \bigcup_{\sigma \in \mathfrak{S}_{\leq 2}^\flat } \tilde{\Lambda}^{(3)}(\sigma)
\end{equation}
for some $\mathfrak{S}_{\leq 2}^\flat$, where all $\tilde{\Lambda}^{(3)}(\sigma)$ have index equal to $3$. The covering \eqref{2091} is impossible: by Lemmas \ref{length3} and \ref{length4} it has no associated minimal covering. 

\item If $Z_{\tilde{\mathbf{t}}, \text{top}}(3) = 5$, then we have $p_{1, \sigma}(\tilde{\bf t}) \equiv p_{2, \sigma}(\tilde{\bf t}) \equiv 0 \bmod 3$ for every $\sigma \in \mathfrak S^\flat$. Since $9 \mid d_2D/g_{2, 4}$ in this case, we have the inclusion $\Lambda(\sigma) \subseteq \tilde{\Lambda}^{(9)}(\sigma)$ and, after division by $3$, we write the equation defining $\tilde{\Lambda}^{(9)}(\sigma)$ ($\sigma \in \mathfrak S^\flat$) as
$$
\tilde{\Lambda}^{(9)}(\sigma): \widetilde{p_{1, \sigma}}(\tilde{\bf t}) x_1 + g_{2, 4} \widetilde{p_{2, \sigma}}(\tilde{\bf t}) x_2 \equiv 0 \bmod 3
$$
with the notations of Lemma \ref{messfor9}. Now the proof is similar to the proof given in the two items above by replacing the pair of integers $(Z_{\tilde{\mathbf{t}}, \text{top}}(3), n_{\tilde{\mathbf{t}}, \text{top}}(3))$ by $(\widetilde{Z_{\tilde{\mathbf{t}}, \text{top}}}(9), \widetilde{n_{\tilde{\mathbf{t}}, \text{top}}}(9))$, and the first part of \eqref{3335} by Lemma \ref{messfor9}. Here also we show that this case is impossible.
\end{itemize}

\noindent We investigated all the cases related to $v_3(g_{2, 4}) = 0$. They all lead to impossibilities.

\subsubsection{\texorpdfstring{Case where $v_3(g_{2, 4}) \geq 2$}{Case where g24 has large 3-adic valuation}} 
\label{2127bis} 
These conditions imply
$$
9\mid d_2 D, \quad (t_1, t_3) \not \equiv (0, 0) \bmod 3, \quad D = 2 \quad \text{ and } \quad 3 \nmid \nu,
$$
the two last ones being a consequence of Step \ref{d2odd}. We now work with the bottom equations of \eqref{generalnotation}. By an explicit study of this bottom equation, we see that, after factoring $g_{2, 4}$ out of $P_{2, \sigma}$, one has the inclusion $\Lambda(\sigma) \subseteq \tilde M^{(9)}(\sigma)$, where $\tilde M^{(9)}(\sigma)$ is the lattice defined by 
\begin{equation}
\label{20999}
\tilde M^{(9)}(\sigma) : P_{1, \sigma}(\tilde{\bf t}) x_1 \equiv 0 \bmod 9 \quad \text{ for } \sigma \in \mathfrak S^\flat.
\end{equation}
Actually, $P_{1, \sigma}(\tilde{\bf t})$ is a polynomial in the variables $t_1$ and $t_3$ only. Since $(t_1, t_3) \not\equiv (0, 0) \bmod 3$, we apply Lemma \ref{3467*} in the form of the inequality \eqref{e3Pbound*} to state that there is at most one $\sigma \in \mathfrak S^\flat$ such that $P_{1, \sigma}(\tilde{\bf t}) \equiv 0 \bmod 9$. Let $\sigma_0$ denote this potential element. Then for every $\sigma \in \mathfrak S^\flat$, $\sigma \neq \sigma_0$, we have the inclusion $\Lambda(\sigma) \subseteq \Lambda_3$. Then the main covering \eqref{maincovering} reads 
$$
\Z^2 = \Lambda_2 \cup \Lambda_3 \cup \Lambda^?.
$$
Such a covering is ruled out by Lemma \ref{length3}.

\subsubsection{\texorpdfstring{Case where $v_3(g_{2, 4})=1$}{Case where g24 has large 3-adic valuation}} 
The conditions now are 
$$ 
3 \Vert g_{2, 4}, \quad 9 \mid d_2 D, \quad (t_1, t_3) \not \equiv (0, 0) \bmod 3, \quad D = 2 \quad \text{ and } \quad 3 \nmid \nu,
$$
the two last ones being a consequence of Step \ref{d2odd}. We again work with the bottom equation of \eqref{generalnotation}. For every $\sigma \in \mathfrak S^\flat$, one has the inclusion $\Lambda(\sigma) \subseteq \tilde M^{(3)}(\sigma)$ where $\tilde M^{(3)}(\sigma)$ is the lattice defined by 
$$
\tilde M^{(3)}(\sigma) : P_{1, \sigma}(\tilde{\bf t}) x_1 \equiv 0 \bmod 3.
$$
We investigate several cases
\begin{itemize}
\item if $t_1\not \equiv t_3 \bmod 3$, then we check by hand that there exists at most one $\sigma \in \mathfrak S^\flat$ such that $P_{1, \sigma}(\tilde{\bf t}) \equiv 0 \bmod 3$. For the remaining $\sigma \in \mathfrak S^\flat$ one has $\tilde M^{(3)}(\sigma ) = \Lambda_3$ and the main covering \eqref{maincovering} leads to the equality
$$
\Z^2 = \Lambda_2 \cup \Lambda_3 \cup \Lambda^?,
$$ 
which is absurd due to Lemma \ref{length3}.
 
\item if $t_1 \equiv t_3 (\not \equiv 0) \bmod 3$, we return to the enveloping lattices $\tilde{\Lambda}^{(3)}(\sigma)$ $(\sigma \in \mathfrak S^\flat)$, which are defined by
$$
\tilde{\Lambda}^{(3)}(\sigma) : p_{1, \sigma}(\pm 1, \tilde t_2, \pm 1, \tilde t_4) x_1 \equiv 0 \bmod 3.
$$
We verify that if $\tilde t_2 \not \equiv \tilde t_4 \bmod 3$, then the equation of each $\tilde{\Lambda}^{(3)}(\sigma)$ is equivalent to $x_1 \equiv 0 \bmod 3$ from which we deduce the impossible equality
$$
\Z^2 = \Lambda_2 \cup \Lambda_3.
$$
\item if $t_1\equiv t_3 (\not \equiv 0) \bmod 3$ and if $\tilde t_2 \equiv \tilde t_4 (\not \equiv 0) \bmod 3$, we utilize the inclusion $\Lambda(\sigma) \subseteq \tilde M^{(9)}(\sigma)$. Note that the defining equation of $\tilde M^{(9)}(\sigma)$ is given by
$$
\tilde M^{(9)}(\sigma) : P_{1, \sigma}(\tilde{\bf t}) x_1\equiv 0 \bmod 9 \quad \text{ for } \sigma \in \mathfrak S^\flat,
$$
since $g_{2, 4} P_{2, \sigma}(\tilde{\bf t}) \equiv 0 \bmod 9$ in this case. We recognize the situation already treated in \S \ref{2127bis}, equation \eqref{20999}, leading to an impossible covering.
\end{itemize}

\noindent The proof of Step \ref{9doesnotdividebis} is complete. We rephrase the set described in Step \ref{9doesnotdividebis} in the form
\begin{equation}
\label{4elements}
(d_2, D, \nu) \in \left\{(1, 2, 2),\, (1, 2, 4),\, (3, 2, 2),\, (3, 2, 4)\right\}
\end{equation}
that we continue to shrink.

\subsection{\texorpdfstring{The case $(d_2, D, \nu) = (1, 2, 4)$ is impossible}{The case (d2, D, nu) = (1, 2, 4) is impossible}}
So suppose that we have $(d_2, D, \nu) = (1, 2, 4)$ in order to arrive at a contradiction. We consider the dual covering \eqref{dualcovering}, where equations for $\Gamma (\sigma)$ are given in Lemma \ref{Lemma6.1*}. In particular, we have
\begin{equation}
\label{2236*}
\Gamma ({\rm id}) = \{(x_1, x_2) : x_2 \equiv 0 \bmod 2\}.
\end{equation}
This lattice has index $2$ and by \eqref{gammaminimal} all the $\Gamma (\sigma)$ have an index $\geq 2$. With the above values of $(d_2, D, \nu)$ we have $D\nu = 8$ and $d_2\nu = 4$. So every system of equations defining $\Gamma (\sigma)$ is equivalent to
\begin{equation}
\label{2281*}
\Gamma (\sigma) : 
\begin{cases} p_{2, \sigma}({\bf t}) x_2 \equiv 0 \bmod 2\\
P_{2, \sigma}({\bf t}) x_2 \equiv 0 \bmod 2.
\end{cases}
\end{equation}
Since $\Gamma(\sigma)$ is a proper lattice, at least one of the integers $p_{2, \sigma}({\bf t})$ and $P_{2, \sigma}({\bf t})$ is odd. This leads to the equality 
$$
\Gamma (\sigma) = \{(x_1, x_2) : x_2 \equiv 0 \bmod 2\} \quad \text{ for } \sigma \in \mathfrak S^\flat.
$$
Upon combining with \eqref{2236*}, we deduce that the covering \eqref{dualcovering} is not possible.

The relation \eqref{4elements} is thus strengthened to
$$
(d_2, D, \nu) \in \left\{(1, 2, 2),\, (3, 2, 2),\, (3, 2, 4)\right\}.
$$

\subsection{\texorpdfstring{The case $(d_2, D, \nu) = (3, 2, 4)$ is impossible}{The case (d2, D, nu) = (3, 2, 4) is impossible}} 
This case generalizes what we just did. We now have $d_2 \nu =12$ and the $\Gamma (\sigma)$ ($\sigma \in \mathfrak S^\flat$) are defined by congruences modulo 12. These lattices are proper, since $\Gamma({\rm id})$ is proper, but this does not imply that they remain proper, when one replaces the modulus $12$ by $4$. The equality \eqref{2281*} is no more true, but we still have the inclusion
$$
\Gamma (\sigma) \subseteq \tilde\Gamma^{(2)}(\sigma) : \begin{cases} p_{2, \sigma}({\bf t}) x_2 \equiv 0 \bmod 2\\
P_{2, \sigma}({\bf t}) x_2 \equiv 0 \bmod 2,
\end{cases}
$$
where $p_{2, \sigma}({\bf t})$ and $P_{2, \sigma}({\bf t})$ are defined by \eqref{generalnotation}. We have the inclusion $\Gamma(\sigma) \subseteq \{(x_1, x_2) : x_2 \equiv 0 \bmod 2\}$ as soon as at least one of the integers $p_{2, \sigma}({\bf t})$ and $P_{2, \sigma}({\bf t})$ is odd. By a direct verification, we see that this congruence modulo $2$ is verified for each $\sigma \in \mathfrak S^\flat$ and for each quadruple $\bf t$ such that $t_1t_4 - t_2t_3$ is odd. Since $3 = d_2 = |t_1t_4 - t_2t_3|$ is indeed odd for us, the dual covering \eqref{dualcovering} is the following nonsense 
$$
\Z^2 = \{(x_1, x_2) : x_2 \equiv 0 \bmod 2\},
$$
and the case $(d_2, D, \nu) = (3, 2, 4)$ is impossible. So we know that necessarily 
$$
(d_2, D, \nu) \in \left\{(1, 2, 2),\, (3, 2, 2)\right\}.
$$

\subsection{\texorpdfstring{Study when $(d_2, D, \nu) = (3, 2, 2)$}{Study when (d2, D, nu) = (3, 2, 2)}} 
We have $g_{2, 4} \in \{1, 3\}$ in our current situation.

\subsubsection{\texorpdfstring{When $g_{2, 4} =1$}{When g24 = 1}} 
\label{2278bis} 
We have $3\mid d_2D/g_{2, 4}$, so we can use the $3$--enveloping lattice $\tilde{\Lambda}^{(3)}(\sigma)$ to exploit the inclusion
$$
\Lambda(\sigma) \subseteq \tilde{\Lambda}^{(3)}(\sigma) \quad \text{ for }  \sigma \in \mathfrak S^\flat,
$$
and the equality $\Lambda({\rm id}) = \Lambda_2$. Our discussion is again based on the possible values of the function $Z_{\mathbf t, \mathrm{ top}}(3)$ that we find in Lemma \ref{lZxnx}, equation \eqref{3335}, in order to understand the behavior of the equations \eqref{1283}. 

\vskip .2cm
$\bullet$ {\bf The equality $Z_{\mathbf t, \mathrm{top}}(3) = 0$ is impossible.}
Indeed, the first inequality of \eqref{3335} and the covering \eqref{maincovering} would lead to a covering of $\Z^2$ by four lattices, one with index 2 (corresponding to $\Lambda_2$), and at most three with index equal to $3$ (corresponding to $\tilde{\Lambda}^{(3)}(\sigma)$). Such a covering is impossible as a consequence of Lemmas \ref{length3} and \ref{length4}.

\vskip .2cm
$\bullet$ {\bf The equality $Z_{\mathbf t, \mathrm{top}}(3) = 1$ is impossible.} 
Now we are led to a covering of $\Z^2$ by at most four lattices: one with index 2 (corresponding to $\Lambda_2$), and at most two with index equal to $3$ and finally some $\Lambda^?$. This is also impossible by Lemmas \ref{length3} and \ref{length4}.

\vskip .2cm
$\bullet$ {\bf The situation when $Z_{\mathbf t, \mathrm{top}}(3) = 5$.} 
By \eqref{eBadt}, we know that $\bf t = \tilde{\bf t}$ belongs to the exceptional set $\mathcal T_3$ containing six elements. So, by definition, we have $p_{1, \sigma}({\bf t}) \equiv p_{2, \sigma}({\bf t}) \equiv 0 \bmod 3$ for all $\sigma \in \mathfrak S^\flat$. We check (or, alternatively, we apply Remark \ref{930}) that we also have the equalities $P_{1, \sigma}({\bf t}) \equiv P_{2, \sigma}({\bf t}) \equiv 0 \bmod 3.$ Then by Remark \ref{representativematrix} we deduce that the representative matrix of $\sigma \in \mathfrak S^\flat$ belongs to $\GL$. This gives the inclusion
\begin{equation}
\label{2294bis}
{\rm Aut}(G_2, \Q) \subset \GL,
\end{equation}
which implies \eqref{supper} of Theorem \ref{sourcebis}. Since $(D, \nu) = (2, 2)$ holds in both cases, the matrix $\gamma$ defined in \eqref{defgamma} has the value
\begin{equation}
\label{2001}
\gamma = \begin{pmatrix} 2 & 0 \\ 0 &1
\end{pmatrix},
\end{equation}
which leads to $G_1= G_2^\dag$. This is precisely the second part of \eqref{G1= G2bis}. So we recovered the conclusion of Theorem \ref{sourcebis} in this case.
 
\subsubsection{\texorpdfstring{When $g_{2, 4} = 3$}{When g24 equals 3}} 
Recall that the $t_i$ are coprime. So we now have the condition $g_{1, 3} = 1$. We consider the bottom equations of \eqref{generalnotation}. Since $(d_2, D, \nu) = (3, 2, 2)$, we have the inclusion $\Lambda(\sigma) \subseteq \tilde M^{(3)}(\sigma)$, where 
$$
\tilde M^{(3)}(\sigma) : P_{1, \sigma}(\tilde{\bf t}) x_1 \equiv 0 \bmod 3.
$$
In this case we get the inclusion \eqref{2294bis} precisely when $t_1 \equiv t_3 \bmod 3$, and all other cases lead to an invalid covering. Indeed, we get $\tilde M^{(3)}(\sigma) \subseteq \{(x_1, x_2) : x_1 \equiv 0 \bmod 3\}$ for all but one $\sigma$ if $t_1 \not \equiv t_3 \bmod 3$. The remainder of the proof is identical to \S \ref{2278bis}.

\subsection{\texorpdfstring{Study when $(d_2, D, \nu) = (1, 2, 2)$}{Study when (d2, D, nu) = (1, 2, 2)}} 
This is the easiest case. By Remark \ref{representativematrix} we deduce that \eqref{2294bis} holds and finally the matrix $\gamma$ (see \eqref{defgamma}) here also has the value given in \eqref{2001}, which leads to $G_1= G_2^\dag$. This is the second part of \eqref{G1= G2bis}. So we recovered the conclusion of Theorem \ref{sourcebis} in this case.

The proof of Theorem \ref{sourcebis} is complete.

\appendix
\section{Algorithms}
\label{a1}
Consider the set
\[
\mathcal{S} = \{(L_1, \dots, L_6) : L_1 \cup \dots \cup L_6 = \Z^2, L_i \text{ is a subgroup of } \Z^2 \text{ for all } i\}.
\]
It will be most convenient to consider general subgroups of $\Z^2$, i.e. we allow the rank of $L_i$ to be smaller than $2$. We turn the set of $6$-tuples $(L_1, \dots, L_6)$ in a partial order by defining
\[
(L_1, \dots, L_6) \preccurlyeq (K_1, \dots, K_6)
\]
if there exists a permutation $\sigma \in S_6$ such that
\[
L_i \subseteq K_{\sigma(i)}.
\]
In particular, this gives a partial order on the set $\mathcal{S}$ by restriction. Recall that $\mathcal{S}$ consists of infinitely many elements. A good way to describe them is as follows.

\begin{theorem}
There are finitely many minimal elements in $\mathcal{S}$ under the ordering $\preccurlyeq$.
\end{theorem}

Write $\mathcal{T}$ for this finite set. Then one may reconstruct $\mathcal{S}$ as
\[
\mathcal{S} = \{(L_1, \dots, L_6) : t \preccurlyeq (L_1, \dots, L_6) \text{ for some } t \in \mathcal{T}\}.
\]
The goal of this section is to give a constructive proof of the above theorem. In fact, we will use an algorithm to compute a finite superset of $\mathcal{T}$. We will also show the following result in the process.

\begin{theorem}
There exists a finite set of points $P \subseteq \Z^2$ with the following property. Suppose that $(L_1, \dots, L_6)$ satisfies
\[
P \subseteq L_1 \cup \dots \cup L_6.
\]
Then we have $L_1 \cup \dots \cup L_6 = \Z^2$.
\end{theorem}

\noindent The list $P$ will be completely explicit.

\subsection{Finding covers}
The goal of this subsection is to give and explain the algorithm that computed Table \ref{table1}. We will first give an important subroutine. All our algorithms are implemented with SageMath.

\subsubsection{Subroutine for covering}
We will now give an algorithm that accomplishes the following task.

\vskip .3cm

\noindent Routine isCover(L). \\
\noindent Input: a list $(L_i)_{1 \leq i \leq n}$, where each $L_i$ is a subgroup of $\Z^2$. \\
\noindent Output: return True if and only if $\bigcup_{1 \leq i \leq n} L_i = \Z^2$.

\begin{verbatim}
V = FreeModule(ZZ, 2)
e1 = vector((1, 0))
e2 = vector((0, 1))

def isCover(L):
    boolean = True
    for i in range(len(L)):
        if L[i].rank() == 2:
            boolean = False
    
    if boolean:
        return False
                    
    W = V.span([e1, e2])    
    for i in range(len(L)):
        if L[i].rank() == 2:
            W = W.intersection(L[i])
    
    genList = W.gens()
    gen1 = genList[0]
    gen2 = genList[1]
    a = gen1[0]
    b = gen2[1]
    
    for i in range(a):
        for j in range(b):
            boolean = False
            for k in range(len(L)):
                v = vector((i, j))
                if L[k].rank() == 2 and v in L[k]:
                    boolean = True
            if not boolean:
                return False

    return True
\end{verbatim}

\begin{lemma}
The subroutine isCover(L) outputs true if and only if $\bigcup_{1 \leq i \leq n} L_i = \Z^2$.
\end{lemma}

\begin{proof}
We have
\[
\bigcup_{1 \leq i \leq n} L_i \neq \Z^2,
\]
if the rank of $L_i$ is smaller than $2$ for all $i$. This is checked in the first seven lines of the algorithm. Write $S$ for the subset of $\{1, \dots, n\}$ such that the rank of $L_i$ equals $2$. Then we observe that
\[
\bigcup_{1 \leq i \leq n} L_i = \Z^2 \Longleftrightarrow \bigcup_{i \in S} L_i = \Z^2.
\]
So it suffices to show that isCover(L) outputs true if and only if $\bigcup_{i \in S} L_i = \Z^2$. Now we check this last condition as follows. Define
\[
W := \bigcap_{i \in S} L_i.
\]
Since all the $L_i$ are of rank $2$, we see that $W$ is of finite index in $\Z^2$. The generators of $W$ are given by SageMath in row echelon form with the first generator being $(a, 0)$ for some $a$, and the second generator being $(\ast, b)$ for some $b$. From this, we see that a fundamental domain for $\Z^2/W$ is
\[
\mathcal{F} = \{(i, j) : 0 \leq i < a, 0 \leq j < b\}.
\]
By this we mean that
\[
\bigcup_{f \in \mathcal{F}} f + W = \Z^2,
\]
and $(f + W) \cap (g + W) = \{0\}$ for all distinct $f, g \in \mathcal{F}$. By construction of $W$, we have that
\[
\bigcup_{i \in S} L_i = \Z^2 \Longleftrightarrow \mathcal{F} \subseteq \bigcup_{i \in S} L_i.
\]
This last condition is checked at the end of the algorithm.
\end{proof}

\subsubsection{Main algorithm}
We will now give the main recursive algorithm. We define the following list
\begin{multline}
\label{eP}
P = [(1, 0), (0, 1), (1, 1), (1, -1), (1, 2), (2, 1), (1, 3), (3, 1), (1, -3), (3, -1), (1, 4), (2, 3), (3, 2), \\
(4, 1), (1, -4), (2, -3), (3, -2), (4, -1), (5, 1), (1, 5), (5, -1), (1, -5), (1, 6), (2, 5), (3, 4), (4, 3), \\
(5, 2), (6, 1), (1, 7), (3, 5), (5, 3), (7, 1), (1, 8), (2, 7), (4, 5), (5, 4), (7, 2), (8, 1), (1, 9), (3, 7), \\
(7, 3), (9, 1), (1, 10), (2, 9), (3, 8), (4, 7), (5, 6), (6, 5), (7, 4), (8, 3), (9, 2), (10, 1), (1, 11), (5, 7), \\
(7, 5), (11, 1), (1, 12), (2, 11), (3, 8), (4, 7), (5, 8), (6, 7), (7, 6), (8, 5), (9, 4), (10, 3), (11, 2), \\
(12, 1), (1, 13), (3, 11), (5, 9), (9, 5), (11, 3), (13, 1), (1, 14), (2, 13), (4, 11), (7, 8), (8, 7), (11, 4), \\
(13, 2), (14, 1), (1, 15), (3, 13), (5, 11), (7, 9), (9, 7), (11, 5), (13, 3), (15, 1), (1, 16), (8, 9), (9, 8), \\
(16, 1), (2, 15), (1, 30), (1, 17), (30, 1), (17, 1)].
\end{multline}
We write $P[j]$ for the $j$-th item of $P$, and $\text{len}(P)$ for the length of $P$.

\vskip .3cm

\noindent Routine FindLattice(L, index).\\
\noindent Input: a list $(L_i)_{1 \leq i \leq n}$, where each $L_i$ is a subgroup of $\Z^2$, and an integer index satisfying $0 \leq \text{index} < \text{len}(P)$. \\
\noindent Assumptions on input: we assume that
\[
P[j] \in \bigcup_{1 \leq i \leq n} L_i
\]
for all $j \geq 0$ smaller than the variable index and that
\[
\bigcup_{1 \leq i \leq n} L_i \neq \Z^2.
\]
\noindent Output: return a set $\mathcal{T}$ of tuples $(K_i)_{1 \leq i \leq n}$ such that $(L_i)_{1 \leq i \leq n} \preccurlyeq (K_i)_{1 \leq i \leq n}$, $\bigcup_{1 \leq i \leq n} K_i = \Z^2$ and such that the following property holds. Let $(L_i')_{1 \leq i \leq n}$ be a tuple satisfying
\[
\bigcup_{1 \leq i \leq n} L_i' = \Z^2, \quad (L_i)_{1 \leq i \leq n} \preccurlyeq (L_i')_{1 \leq i \leq n}, \quad L_i' \neq \Z^2 \text{ for all } i.
\]
Then there exists $(K_i)_{1 \leq i \leq n} \in \mathcal{T}$ such that $(K_i)_{1 \leq i \leq n} \preccurlyeq (L_i')_{1 \leq i \leq n}$.

\begin{verbatim}
e1 = vector((1, 0))
e2 = vector((0, 1))

def FindLattice(L, index):
    Solutions = []
    v = PointList[index]
    
    LatticeIndex = 0
    while L[LatticeIndex].rank() != 0 and LatticeIndex < len(L) - 1:
        LatticeIndex = LatticeIndex + 1
        
    Boolean = True
    while Boolean:
        Boolean = False
        for i in range(len(L)):
            if v in L[i]:
                Boolean = True
                index = index + 1
                v = PointList[index]
                break

    for i in range(LatticeIndex + 1):
        genList = []
        for j in range(len(L[i].gens())):
            genList = genList + [L[i].gen(j)]
        genList = genList + [v]
        Lnew = L[i].span(genList)
        Lold = L[i]

        if e1 not in Lnew or e2 not in Lnew:
            L[i] = Lnew
            if isCover(L):
                Solutions = Solutions + L
            else:
                Solutions = Solutions + FindLattice(L, index + 1)
        L[i] = Lold
    return Solutions
\end{verbatim}

\noindent Here PointList is the list $P$ as specified in equation (\ref{eP}).

\begin{lemma}
Suppose that $n \leq 6$. Then the algorithm FindLattice(L, index) terminates and gives the claimed output.
\end{lemma}

\begin{proof}
Let us first remark that the algorithm FindLattice(L, index) in principle may crash. This can only happen in the assignment
\begin{verbatim}
v = PointList[index]
\end{verbatim}
in the first while loop, because the variable index may be out of bounds, or in other words the variable index might be at least the length of the list $P$. We will first show that this does not happen for $n \leq 6$, and here we will of course use the specific shape of $P$ in equation (\ref{eP}).

We run the algorithm with $L = (\{0\})_{1 \leq i \leq 6}$ on the computer, and we find that the algorithm terminates without crashing. Now suppose that
\begin{align}
\label{eForceP}
P \subseteq \bigcup_{1 \leq i \leq 6} L_i
\end{align}
for some subgroups $L_i \subseteq \Z^2$. Then we claim that
\[
\bigcup_{1 \leq i \leq 6} L_i = \Z^2.
\]
Suppose that the claim is false. We will show that then the algorithm crashes. Indeed, take such a list $(L_i)_{1 \leq i \leq 6}$. By equation (\ref{eForceP}) we may ensure that we recursively call the routine FindLattice(K, index) with a list $(K_i)_{1 \leq i \leq 6}$ satisfying
\[
(K_i)_{1 \leq i \leq 6} \preccurlyeq (L_i)_{1 \leq i \leq 6}.
\]
But then the algorithm must crash.

From the claim we deduce that the algorithm does not crash on all valid inputs for $n \leq 6$. Furthermore, it is also clear that the algorithm terminates if it does not crash, since at each recursive call we increase the variable index by $1$.

It remains to show that the algorithm gives the claimed output. Take a tuple $(L_i')_{1 \leq i \leq 6}$ satisfying
\[
\bigcup_{1 \leq i \leq 6} L_i' = \Z^2, \quad (L_i)_{1 \leq i \leq 6} \preccurlyeq (L_i')_{1 \leq i \leq 6}, \quad L_i' \neq \Z^2 \text{ for all } i.
\]
To find our tuple $(K_i)_{1 \leq i \leq 6}$, we simply enforce that $(K_i)_{1 \leq i \leq 6} \preccurlyeq (L_i')_{1 \leq i \leq 6}$ at each recursive step.
\end{proof}

\subsubsection{Reducing the list}
The above algorithm gives an output of $6131$ elements. However, many of those are not minimal under the ordering $\preccurlyeq$. We prune them to a list of only $54$ elements using the following algorithm.

\begin{verbatim}
n = 6
V = FreeModule(ZZ,2)
TrivialModule = V.zero_submodule()
e1 = vector((1, 0))
e2 = vector((0, 1))

Answer = FindLattice([TrivialModule for i in range(n)], 0)
RealAnswer = [[TrivialModule for i in range(n)] for 
	j in range(Integer(len(Answer)/n))]

for i in range(Integer(len(Answer)/n)):
    RealAnswer[i] = prune([Answer[n * i + j] for j in range(n)])
    
ShortList = []
boolean = True
for i in range(len(RealAnswer)):
    for j in range(len(ShortList)):
        boolean = True
        if containment(ShortList[j], RealAnswer[i]):
            boolean = False
            break
    if boolean:
        ShortList.append(RealAnswer[i])
        
print(ShortList)
\end{verbatim}

The subroutine prune checks if one can remove any element of the list $(L_i)_{1 \leq i \leq n}$ such that the result still has union $\Z^2$. If so, it removes such elements from the list.

\begin{verbatim}
def prune(L):
    for i in range(len(L)):
        Lold2 = L[i]
        L[i] = TrivialModule
        if not isCover(L):
            L[i] = Lold2
            
    zeroCount = 0
    for i in range(len(L)):
        if L[i] == TrivialModule:
            zeroCount = zeroCount + 1
            
    for i in range(len(L) - zeroCount):
        if L[i] == TrivialModule:
            for j in range(len(L) - zeroCount, len(L)):
                if L[j] != TrivialModule:
                    L[i] = L[j]
                    L[j] = TrivialModule
                    break
                    
    return(L)
\end{verbatim}

Finally, the subroutine containment returns true if and only if $L_1 \preccurlyeq L_2$. We remark that SageMath considers permutations on $\{1, \dots, n\}$, while indices of lists start at $0$. Therefore we use $\sigma(i + 1) - 1$ to get a permutation on $\{0, \dots, n - 1\}$.

\begin{verbatim}
def containment(L1, L2):
    for iter in range(len(L1)):
        if L1[iter] == TrivialModule:
            break

    G = SymmetricGroup(iter + 1)    
    for sigma in G:
        boolean = True
        for i in range(len(L1)):
            genList = L1[i].gens()
            for j in range(len(genList)):
                if genList[j] not in L2[sigma(i + 1) - 1]:
                    boolean = False
            
        if boolean:
            return True

    return False
\end{verbatim}

\subsection{Gr\"obner bases}
Here we shall provide full proofs of Lemma \ref{lSX3.2} and Lemma \ref{lTripleVanishing}. The following routine will be essential.

\vskip .3cm

\noindent Routine Groebner(L).\\
\noindent Input: a list $L$ of polynomials in $\mathbb{Z}[\mathbf{x}]$. \\
\noindent Output: return True if and only if $3$ is in the Gr\"obner basis of the ideal generated by the elements in $L$.

\begin{verbatim}
def Groebner(L):
    I = ideal(L)
    B = I.groebner_basis()

    if 3 in B:
        return True
    else:
        return False
\end{verbatim}

\noindent We will now define the polynomials from Lemma \ref{lLatticeEquationsbis} in SageMath.

\begin{verbatim}
R.<t1, t2, t3, t4, u1, u2, u3, u4> = ZZ['t1, t2, t3, t4, u1 ,u2, u3, u4']
R1 = t1 * t2 + t2 * t3 + t3 * t4
R2 = t2 * t2 + t2 * t4 + t4 * t4

Rsquared1 = t1 * t2 + t1 * t4 + t3 * t4
Rsquared2 = t2 * t2 + t2 * t4 + t4 * t4

S1 = t3 * t4 - t1 * t2
S2 = t4 * t4 - t2 * t2

RS1 = t1 * t2 + t1 * t4 + t2 * t3
RS2 = t2 * t2 + 2 * t2 * t4

RsquaredS1 = t1 * t4 + t2 * t3 + t3 * t4
RsquaredS2 = t4 * t4 + 2 * t2 * t4

coprimet2t4 = u2 * t2 + u4 * t4 - 1
coprimet1t3 = u1 * t1 + u3 * t3 - 1
coprime = u1 * t1 + u2 * t2 + u3 * t3 + u4 * t4 - 1
\end{verbatim}

\subsubsection{Proof of Lemma \ref{lSX3.2}}
\label{aSX3.2}
In order to prove Lemma \ref{lSX3.2} in the other cases, we start by defining the polynomials in the bottom equations. We will only need the polynomial appearing in front of $x_1$.

\begin{verbatim}
botR = t1 * t1 + t1 * t3 + t3 * t3
botRsquared = t1 * t1 + t1 * t3 + t3 * t3
botS = t1 * t1 - t3 * t3
botRS = t1 * t1 + 2 * t1 * t3
botRsquaredS = t3 * t3 + 2 * t1 * t3
\end{verbatim}

\noindent We may now prove Lemma \ref{lSX3.2} in full generality.

\begin{verbatim}
print(Groebner([R1, R2, Rsquared1, Rsquared2, botR, botRsquared, coprime]))
print(Groebner([R1, R2, S1, S2, botR, botS, coprime]))
print(Groebner([R1, R2, RS1, RS2, botR, botRS, coprime]))
print(Groebner([R1, R2, RsquaredS1, RsquaredS2, 
 botR, botRsquaredS, coprime]))
print(Groebner([Rsquared1, Rsquared2, S1, S2, botRsquared, botS, coprime]))
print(Groebner([Rsquared1, Rsquared2, RS1, RS2, 
 botRsquared, botRS, coprime]))
print(Groebner([Rsquared1, Rsquared2, RsquaredS1, RsquaredS2, 
 botRsquared, botRsquaredS, coprime]))
print(Groebner([S1, S2, RS1, RS2, botS, botRS, coprime]))
print(Groebner([S1, S2, RsquaredS1, RsquaredS2, 
 botS, botRsquaredS, coprime]))
print(Groebner([RS1, RS2, RsquaredS1, RsquaredS2, 
 botRS, botRsquaredS, coprime]))
\end{verbatim}

Once more we see that the first prompt returns False, and this corresponds precisely to the case that $\sigma$ and $\sigma'$ both have order $3$.

\subsubsection{Proof of Lemma \ref{lTripleVanishing}}
\label{aTripleVanishing}
We will now prove Lemma \ref{lTripleVanishing}.

\begin{verbatim}
print(Groebner([R1, Rsquared1, S1, coprimet1t3, coprimet2t4]))
print(Groebner([R1, Rsquared1, RS1, coprimet1t3, coprimet2t4]))
print(Groebner([R1, Rsquared1, RsquaredS1, coprimet1t3, coprimet2t4]))
print(Groebner([R1, S1, RS1, coprimet1t3, coprimet2t4]))
print(Groebner([R1, S1, RsquaredS1, coprimet1t3, coprimet2t4]))
print(Groebner([R1, RS1, RsquaredS1, coprimet1t3, coprimet2t4]))
print(Groebner([Rsquared1, S1, RS1, coprimet1t3, coprimet2t4]))
print(Groebner([Rsquared1, S1, RsquaredS1, coprimet1t3, coprimet2t4]))
print(Groebner([Rsquared1, RS1, RsquaredS1, coprimet1t3, coprimet2t4]))
print(Groebner([S1, RS1, RsquaredS1, coprimet1t3, coprimet2t4]))
\end{verbatim}

We remark that in this case the last prompt returns False, which is caused by the exceptional case $\{T_2^{-1}ST_2, T_2^{-1}RST_2, T_2^{-1}R^2ST_2\}$. However, we see that $t_1^2 + t_1t_3 + t_3^2$ is in the Gr\"obner basis for the exceptional case. If $\gcd(t_1, t_3, p) = 1$, then this can only be zero modulo $p$ if $t_1$ and $t_3$ are both invertible. But then we see that $x^2 + x + 1 \equiv 0 \bmod p$ has a solution. For a prime $p$, this implies that $p \equiv 1 \bmod 3$. For $p = 9$, we directly check that there are no solutions to $x^2 + x + 1 \equiv 0 \bmod 9$.

\subsection{Brute force searches}
\label{aBruteForce}
Here we will prove Lemma \ref{lZxnx}.

\subsubsection{Computing the number of equivalence classes}
Let $\mathcal{S}$ be a subset of $(\Z/n\Z)^2$. We can decompose $\mathcal{S}$ as $\mathcal{S}_{<n}$ and $\mathcal{S}_n$ of elements of order respectively smaller than $n$ and exactly equal to $n$. Write $m$ for the number of equivalence classes in $\mathcal{S}_n$ under the equivalence relation $\sim$ defined in Definition \ref{dZn}. The next code computes
\[
|\mathcal{S}_{<n}| + m.
\]

\begin{verbatim}
def Total(L, n):
    total = 0
    l = len(L)
    
    for i in range(l):
        boolean = True
        point = L[i]
        if gcd(point[0], n) != 1 and gcd(point[1], n) != 1:
            total = total + 1
        else:
            for j in range(i):
                for k in range(n):
                    if gcd(k, n) == 1 and k * point[0] == L[j][0] and 
                     k * point[1] == L[j][1]:
                        boolean = False
            if boolean:            
                total = total + 1          
    return total
\end{verbatim}

\subsubsection{\texorpdfstring{Proof of Lemma \ref{lZxnx} for $3$, $4$ and $5$}{Proof of Lemma 4.12 for 3, 4, 5}}
We are now ready to prove Lemma \ref{lZxnx}. Consider the following code.

\begin{verbatim}
def BruteForce(n):
 for a in [0..n - 1]:
 for b in [0..n - 1]:
 for c in [0..n - 1]:
 for d in [0..n - 1]:
 x1 = Mod(c * d - a * b, n)
 y1 = Mod(d * d - b * b, n)
 x2 = Mod(a * b + a * d + b * c, n)
 y2 = Mod(b * b + 2 * b * d, n)
 x3 = Mod(a * d + b * c + c * d, n)
 y3 = Mod(d * d + 2 * b * d, n)
 x4 = Mod(a * b + b * c + c * d, n)
 y4 = Mod(b * b + b * d + d * d, n)
 x5 = Mod(a * b + a * d + c * d, n)
 y5 = Mod(b * b + b * d + d * d, n)
 
 L = [vector((1, 0)), vector((x1, y1)), 
 vector((x2, y2)), vector((x3, y3)), 
 vector((x4, y4)), vector((x5, y5))]
 count = Total(L, n)
 
 if count > 3 and gcd(n, gcd(b, d)) == 1:
 print(count)
 print('(' + str(x1) + ',' + str(y1) + ')' + ' AND ' 
 + '(' + str(x2) + ',' + str(y2) + ')' + ' AND ' 
 + '(' + str(x3) + ',' + str(y3) + ')' + ' AND ' 
 + '(' + str(x4) + ',' + str(y4) + ')' + ' AND ' 
 + '(' + str(x5) + ',' + str(y5) + ')' 
 + ' AND (1, 0)')
 print(a, b, c, d)
\end{verbatim}

We run BruteForce(3), BruteForce(4) and BruteForce(5). If $\mathbf{t} = (t_1, t_2, t_3, t_4)$ satisfies $\gcd(t_2, t_4, x) = 1$, then the variable count is exactly equal to $Z_{\mathbf{t}, \text{top}}(x) + n_{\mathbf{t}, \text{top}}(x)$. Note that output gets printed precisely when the variable count is strictly larger than $3$. We now examine the output for $x \in \{3, 4, 5\}$.

If $x = 5$, then no output gets printed, which proves Lemma \ref{lZxnx} in that case. If $x = 4$, then the variable count can only equal $4$ when $Z_{\mathbf{t}, \text{top}}(4) = 1$ coming from some $\sigma$ with $(p_{1, \sigma}(\mathbf{t}), p_{2, \sigma}(\mathbf{t})) \equiv (2, 0) \bmod 4$. Examining the output, we see that this is indeed the case proving alinea $(i)$. We will prove the contents of alinea $(ii)$ and alinea $(iii)$ in the next subsubsection. If $x = 3$, then Lemma \ref{lZxnx} dictates that the variable count is greater than $3$ exactly when $\mathbf{t}$ lies in $\mathcal{T}_3$ defined in \eqref{eBadt}. In this case we get that $(p_{1, \sigma}(\mathbf{t}), p_{2, \sigma}(\mathbf{t})) \equiv (0, 0) \bmod 3$ for all $\sigma \in \mathfrak{S}^\flat$. This is indeed the case.

\subsubsection{Additional results for 4}
We will now focus on the supplemental results for $x = 4$, i.e. Lemma \ref{lZxnx} alinea $(ii)$ and alinea $(iii)$. For this we use the following routine.

\begin{verbatim}
def check(L):
    badnessCounter = 0
    boolean1 = True
    boolean2 = True
    
    for i in range(len(L)):
        if boolean1 and (L[i] == vector((0, 1)) or L[i] == vector((0, 3))):
            badnessCounter = badnessCounter + 1
            boolean1 = False
        if boolean2 and (L[i] == vector((2, 1)) or L[i] == vector((2, 3))):
            badnessCounter = badnessCounter + 1
            boolean2 = False
        if L[i] == vector((2, 0)) or L[i] == vector((0, 2)) or 
         L[i] == vector((0, 0)) or L[i] == vector((2, 2)):
            badnessCounter = badnessCounter + 1
            
    return badnessCounter
\end{verbatim}

Suppose that $\gcd(t_1, t_3, 2) = 1$. Then badnessCounter is at most $1$ if and only if alinea $(ii)$ and alinea $(iii)$ of Lemma \ref{lZxnx} are true. This motivates us to insert the following code in the main body

\begin{verbatim}
badness = check(L)
if badness > 1 and gcd(2, gcd(a, c)) == 1 and gcd(2, gcd(b, d)) == 1:
    print("ERROR!")
\end{verbatim}

\noindent Running this script produces no output, so alinea $(ii)$ and alinea $(iii)$ of Lemma \ref{lZxnx} hold.

\subsubsection{\texorpdfstring{Proof of Lemma \ref{messfor9} for $9$}{Proof of Lemma 4.13 for 9}}
For the part about the reduction modulo $9$, we consider the following variation of BruteForce(n).

\begin{verbatim}
def BruteForce(n, BadTuple1, BadTuple2, BadTuple3, BadTuple4):
    for a in [0..n - 1]:
        for b in [0..n - 1]:
            for c in [0..n - 1]:
                for d in [0..n - 1]:
                    a1 = 3 * a + BadTuple1
                    b1 = 3 * b + BadTuple2
                    c1 = 3 * c + BadTuple3
                    d1 = 3 * d + BadTuple4
 
                    x1 = Mod((c1 * d1 - a1 * b1)/n, n)
                    y1 = Mod((d1 * d1 - b1 * b1)/n, n)
                    x2 = Mod((a1 * b1 + a1 * d1 + b1 * c1)/n, n)
                    y2 = Mod((b1 * b1 + 2 * b1 * d1)/n, n)
                    x3 = Mod((a1 * d1 + b1 * c1 + c1 * d1)/n, n)
                    y3 = Mod((d1 * d1 + 2 * b1 * d1)/n, n)
                    x4 = Mod((a1 * b1 + b1 * c1 + c1 * d1)/n, n)
                    y4 = Mod((b1 * b1 + b1 * d1 + d1 * d1)/n, n)
                    x5 = Mod((a1 * b1 + a1 * d1 + c1 * d1)/n, n)
                    y5 = Mod((b1 * b1 + b1 * d1 + d1 * d1)/n, n)
 
                    L = [vector((1, 0)), vector((x1, y1)), 
                    vector((x2, y2)), vector((x3, y3)), 
                    vector((x4, y4)), vector((x5, y5))]
                    count = Total(L, n)
 
                    if count > 3:
                        print('(' + str(x1) + ',' + str(y1) + ')' + ' AND ' 
                         + '(' + str(x2) + ',' + str(y2) + ')' + ' AND ' 
                         + '(' + str(x3) + ',' + str(y3) + ')' + ' AND ' 
                         + '(' + str(x4) + ',' + str(y4) + ')' + ' AND ' 
                         + '(' + str(x5) + ',' + str(y5) + ')' 
                         + ' AND (1, 0)')
                        print(a, b, c, d)
                        print(x1, y1, x2, y2, x3, y3, x4, y4, x5, y5)
\end{verbatim}

\noindent We now run BruteForce(3, 0, 1, 0, 1), BruteForce(3, 1, 1, 1, 1), BruteForce(3, 1, 2, 1, 2) to complete the proof of Lemma \ref{messfor9} as the script does not print any output.

\subsubsection{\texorpdfstring{Proof of last part of Lemma \ref{lTripleVanishing}}{Proof of last part of Lemma 4.14}}
\label{a414}
We run a minor variation of the usual brute force script to demonstrate the last part of Lemma \ref{lTripleVanishing}.

\begin{verbatim}
def BruteForce2(n):
    for a in [0..n - 1]:
        for b in [0..n - 1]:
            for c in [0..n - 1]:
                for d in [0..n - 1]:
                    x1 = Mod(c * d - a * b, n)
                    x2 = Mod(a * b + a * d + b * c, n)
                    x3 = Mod(a * d + b * c + c * d, n)
                    x4 = Mod(a * b + b * c + c * d, n)
                    x5 = Mod(a * b + a * d + c * d, n)
                    count = 0
                    if x1 == 0:
                        count = count + 1
                    if x2 == 0:
                        count = count + 1
                    if x3 == 0:
                        count = count + 1
                    if x4 == 0:
                        count = count + 1
                    if x5 == 0:
                        count = count + 1
                    if count > 2 and gcd(a, c) != 0 and gcd(b, d) != 0:   
                        print(a, b, c, d)
                        
BruteForce2(3)
\end{verbatim}

\noindent The output is precisely the set $\mathcal{T}_3^\ast$ in accordance with the last part of Lemma \ref{lTripleVanishing}.

\section{Output of main algorithm}
\label{a2}
Below we give the complete list (up to permutations) of the $54$ minimal coverings of $\Z^2$ requiring 3, 4, 5 or 6 lattices. Each lattice is described by a $\Z$--basis: the matrix 
\[
\begin{pmatrix} 
a & b \\ 
c & d
\end{pmatrix}
\]
in the cell corresponding to the lattice $L_i$ means that $L_i= \Z\cdot \begin{pmatrix} a\\c\end{pmatrix} \oplus \Z \cdot \begin{pmatrix} b \\ d \end{pmatrix}$.

\begin{center}
\begin{longtable}{| c | c | c | c | c | c | c |}
\hline
Entry & $L_1$ & $L_2$ & $L_3$ & $L_4$ & $L_5$ & $L_6$ \\ \hline
1 & $\begin{pmatrix}1 & 0 \\ 0 &2\end{pmatrix}$ & $\begin{pmatrix}2 & 0 \\ 0 & 1\end{pmatrix}$ & $\begin{pmatrix}1 & 0 \\ 1 & 2 \end{pmatrix}$ & $ $ & $ $ & $ $ \\ \hline
 
2 & $\begin{pmatrix}1 & 0 \\ 0 & 2\end{pmatrix}$ & $\begin{pmatrix}4 & 0 \\ 0 & 1\end{pmatrix}$ & $\begin{pmatrix}1 & 0 \\ 1 & 2\end{pmatrix}$ & $\begin{pmatrix}2 & 0 \\ 1 & 2\end{pmatrix}$ & $ $ & $ $ \\ \hline

3 & $\begin{pmatrix}1 & 0 \\ 0 & 2\end{pmatrix}$ & $\begin{pmatrix}8 & 0 \\ 0 & 1\end{pmatrix}$ & $\begin{pmatrix}1 & 0 \\ 1 & 2\end{pmatrix}$ & $\begin{pmatrix}2 & 0 \\ 1 & 2\end{pmatrix}$ & $\begin{pmatrix}4 & 0 \\ 1 & 2\end{pmatrix}$ & $ $ \\ \hline
 
4 & $\begin{pmatrix}1 & 0 \\ 0 & 2\end{pmatrix}$ & $\begin{pmatrix}16 & 0 \\ 0 & 1\end{pmatrix}$ & $\begin{pmatrix}1 & 0 \\ 1 & 2\end{pmatrix}$ & $\begin{pmatrix}2 & 0 \\ 1 & 2\end{pmatrix}$ & $\begin{pmatrix}4 & 0 \\ 1 & 2\end{pmatrix}$ & $\begin{pmatrix}8 & 0 \\ 1 & 2\end{pmatrix}$ \\ \hline

5 & $\begin{pmatrix}1 & 0 \\ 0 & 2\end{pmatrix}$ & $\begin{pmatrix}8 & 0 \\ 0 & 1\end{pmatrix}$ & $\begin{pmatrix}1 & 0 \\ 1 & 2\end{pmatrix}$ & $\begin{pmatrix}2 & 0 \\ 1 & 2\end{pmatrix}$ & $\begin{pmatrix}4 & 0 \\ 1 & 4\end{pmatrix}$ & $\begin{pmatrix}4& 0 \\ 3 & 4\end{pmatrix}$ \\ \hline

6 & $\begin{pmatrix}1& 0 \\ 0 & 2\end{pmatrix}$ & $\begin{pmatrix}4 & 0 \\ 0 & 1\end{pmatrix}$ & $\begin{pmatrix}1 & 0 \\ 1 & 2\end{pmatrix}$ & $\begin{pmatrix}2 & 0 \\ 1 & 4\end{pmatrix}$ & $\begin{pmatrix}2 & 0 \\ 3 & 4\end{pmatrix}$ & $ $ \\ \hline

7 & $\begin{pmatrix}1 & 0 \\ 0 & 2\end{pmatrix}$ & $\begin{pmatrix}4 & 0 \\ 0 & 1\end{pmatrix}$ & $\begin{pmatrix}1 & 0 \\ 1 & 2\end{pmatrix}$ & $\begin{pmatrix}2 & 0 \\ 1 & 4\end{pmatrix}$ & $\begin{pmatrix}2 & 0 \\ 3 & 8\end{pmatrix}$ & $\begin{pmatrix}2 & 0 \\ 7 & 8\end{pmatrix}$ \\ \hline

8 & $\begin{pmatrix}1 & 0 \\ 0 & 2\end{pmatrix}$ & $\begin{pmatrix}4 & 0 \\ 0 & 1\end{pmatrix}$ & $\begin{pmatrix}1 & 0 \\ 1 & 2\end{pmatrix}$ & $\begin{pmatrix}2 & 0 \\ 1& 8\end{pmatrix}$ & $\begin{pmatrix}2 & 0 \\ 3 & 4\end{pmatrix}$ & $\begin{pmatrix}2 & 0 \\ 5 & 8\end{pmatrix}$ \\ \hline

9 & $\begin{pmatrix}1 & 0 \\ 0 & 2\end{pmatrix}$ & $\begin{pmatrix}8 & 0 \\ 0 & 1\end{pmatrix}$ & $\begin{pmatrix}1 & 0 \\ 1 & 2\end{pmatrix}$ & $\begin{pmatrix} 2 & 0 \\ 1 & 4\end{pmatrix}$ & $\begin{pmatrix}2 & 0 \\ 3 & 4\end{pmatrix}$ & $\begin{pmatrix}4 & 0 \\ 1 & 2\end{pmatrix}$ \\ \hline

10 & $\begin{pmatrix}1 & 0 \\ 0 & 2\end{pmatrix}$ & $\begin{pmatrix}6 & 0 \\ 0 & 1\end{pmatrix}$ & $\begin{pmatrix}1 & 0 \\ 1 & 2\end{pmatrix}$ & $\begin{pmatrix}2 & 0 \\ 1 & 3\end{pmatrix}$ & $\begin{pmatrix}2 & 0 \\ 0 & 3\end{pmatrix}$ & $\begin{pmatrix}2 & 0 \\ 2 & 3\end{pmatrix}$ \\ \hline

11 & $\begin{pmatrix}1 & 0 \\ 0 & 4\end{pmatrix}$ & $\begin{pmatrix}2 & 0 \\ 0 & 1\end{pmatrix}$ & $\begin{pmatrix}1 & 0 \\ 1& 2\end{pmatrix}$ & $\begin{pmatrix}1 & 0 \\ 2 & 4\end{pmatrix}$ & $ $ & $ $ \\ \hline

12 & $\begin{pmatrix}1 & 0 \\ 0 & 4\end{pmatrix}$ & $\begin{pmatrix}2 & 0 \\ 0 & 1\end{pmatrix}$ & $\begin{pmatrix}1 & 0 \\ 1 & 2\end{pmatrix}$ & $\begin{pmatrix}1 & 0 \\ 2 & 8\end{pmatrix}$ & $\begin{pmatrix}1 & 0 \\ 6 & 8\end{pmatrix}$ & $ $ \\ \hline

13 & $\begin{pmatrix}1 & 0 \\ 0 & 4\end{pmatrix}$ & $\begin{pmatrix}2 & 0 \\ 0 & 1\end{pmatrix}$ & $\begin{pmatrix}1 & 0 \\ 1 & 2\end{pmatrix}$ & $\begin{pmatrix}1 & 0 \\ 2 & 8\end{pmatrix}$ & $\begin{pmatrix}1 & 0 \\ 6 & 16\end{pmatrix}$ & $\begin{pmatrix}1 & 0 \\ 14 & 16\end{pmatrix}$ \\ \hline

14 & $\begin{pmatrix}1 & 0 \\ 0 & 4\end{pmatrix}$ & $\begin{pmatrix}2 & 0 \\ 0 & 1\end{pmatrix}$ & $\begin{pmatrix}1 & 0 \\ 1 & 2\end{pmatrix}$ & $\begin{pmatrix}1 & 0 \\ 2 & 16\end{pmatrix}$ & $\begin{pmatrix}1 & 0 \\ 6 & 8\end{pmatrix}$ & $\begin{pmatrix}1 & 0 \\ 10 & 16\end{pmatrix}$ \\ \hline

15 & $\begin{pmatrix}1 & 0 \\ 0 & 8\end{pmatrix}$ & $\begin{pmatrix}2 & 0 \\ 0 & 1\end{pmatrix}$ & $\begin{pmatrix}1 & 0 \\ 1 & 2\end{pmatrix}$ & $\begin{pmatrix}1 & 0 \\ 2 & 4\end{pmatrix}$ & $\begin{pmatrix}1 & 0 \\ 4 & 8\end{pmatrix}$ & $ $ \\ \hline

16 & $\begin{pmatrix}1 & 0 \\ 0 & 16\end{pmatrix}$ & $\begin{pmatrix}2 & 0 \\ 0 & 1\end{pmatrix}$ & $\begin{pmatrix}1 & 0 \\ 1 & 2\end{pmatrix}$ & $\begin{pmatrix}1 & 0 \\ 2 & 4\end{pmatrix}$ & $\begin{pmatrix}1 & 0 \\ 4 & 8\end{pmatrix}$ & $\begin{pmatrix}1 & 0 \\ 8 & 16\end{pmatrix}$ \\ \hline

17 & $\begin{pmatrix}1 & 0 \\ 0 & 8\end{pmatrix}$ & $\begin{pmatrix} 2 & 0 \\ 0 & 1\end{pmatrix}$ & $\begin{pmatrix}1 & 0 \\ 1 & 2\end{pmatrix}$ & $\begin{pmatrix}1 & 0 \\ 2 & 4\end{pmatrix}$ & $\begin{pmatrix}1 & 0 \\ 4 & 16\end{pmatrix}$ & $\begin{pmatrix}1 & 0 \\ 12 & 16\end{pmatrix}$ \\ \hline

18 & $\begin{pmatrix}1 & 0 \\ 0 & 6\end{pmatrix}$ & $\begin{pmatrix}2 & 0 \\ 0 & 1\end{pmatrix}$ & $\begin{pmatrix}1 & 0 \\ 1 & 2\end{pmatrix}$ & $\begin{pmatrix}1 & 0 \\ 2 & 6\end{pmatrix}$ & $\begin{pmatrix}1 & 0 \\ 4 & 6\end{pmatrix}$ & $\begin{pmatrix}3 & 0 \\ 0 & 2\end{pmatrix}$ \\ \hline

19 & $\begin{pmatrix}1 & 0 \\ 0 & 8\end{pmatrix}$ & $\begin{pmatrix}2 & 0 \\ 0 & 1\end{pmatrix}$ & $\begin{pmatrix}1 & 0 \\ 1 & 2\end{pmatrix}$ & $\begin{pmatrix}1 & 0 \\ 2 & 8\end{pmatrix}$ & $\begin{pmatrix}1 & 0 \\ 4 & 8\end{pmatrix}$ & $\begin{pmatrix}1 & 0 \\ 6 & 8\end{pmatrix}$ \\ \hline

20 & $\begin{pmatrix}1 & 0 \\ 0 & 3\end{pmatrix}$ & $\begin{pmatrix}3 & 0 \\ 0 & 1\end{pmatrix}$ & $\begin{pmatrix}1 & 0 \\ 1 & 3\end{pmatrix}$ & $\begin{pmatrix}1 & 0 \\ 2 & 3\end{pmatrix}$ & $ $ & $ $ \\ \hline

21 & $\begin{pmatrix}1 & 0 \\ 0 & 4\end{pmatrix}$ & $\begin{pmatrix}4 & 0 \\ 0 & 1\end{pmatrix}$ & $\begin{pmatrix}1 & 0 \\ 1 & 2\end{pmatrix}$ & $\begin{pmatrix}1 & 0 \\ 2 & 4\end{pmatrix}$ & $\begin{pmatrix}2 & 0 \\ 1 & 2\end{pmatrix}$ & $ $ \\ \hline

22 & $\begin{pmatrix}1 & 0 \\ 0 & 4\end{pmatrix}$ & $\begin{pmatrix}8 & 0 \\ 0 & 1\end{pmatrix}$ & $\begin{pmatrix}1 & 0 \\ 1 & 2\end{pmatrix}$ & $\begin{pmatrix}1 & 0 \\ 2 & 4\end{pmatrix}$ & $\begin{pmatrix}2 & 0 \\ 1 & 2\end{pmatrix}$ & $\begin{pmatrix}4 & 0 \\ 1& 2\end{pmatrix}$ \\ \hline

23 & $\begin{pmatrix}1 & 0 \\ 0 & 4\end{pmatrix}$ & $\begin{pmatrix} 4 & 0 \\ 0 & 1\end{pmatrix}$ & $\begin{pmatrix}1 & 0 \\ 1 & 2\end{pmatrix}$ & $\begin{pmatrix}1 & 0 \\ 2 & 8\end{pmatrix}$ & $\begin{pmatrix}2 & 0 \\ 1 & 2 \end{pmatrix}$ & $\begin{pmatrix}1 & 0 \\ 6 & 8\end{pmatrix}$ \\ \hline

24 & $\begin{pmatrix}1 & 0 \\ 0 & 4\end{pmatrix}$ & $\begin{pmatrix} 4 & 0 \\ 0 & 1\end{pmatrix}$ & $\begin{pmatrix} 1 & 0 \\ 1 & 2\end{pmatrix}$ & $\begin{pmatrix}1 & 0 \\ 2 & 4\end{pmatrix}$ & $\begin{pmatrix}2 & 0 \\ 1 & 4\end{pmatrix}$ & $\begin{pmatrix}2 & 0 \\ 3 & 4\end{pmatrix}$ \\ \hline

25 & $\begin{pmatrix}1 & 0 \\ 0 & 8\end{pmatrix}$ & $\begin{pmatrix} 4 & 0 \\ 0 & 1\end{pmatrix}$ & $\begin{pmatrix}1 & 0 \\ 1 & 2\end{pmatrix}$ & $\begin{pmatrix}1 & 0 \\ 2 & 4\end{pmatrix}$ & $\begin{pmatrix}2 & 0 \\ 1 & 2\end{pmatrix}$ & $\begin{pmatrix}1 & 0 \\ 4 & 8\end{pmatrix}$ \\ \hline

26 & $\begin{pmatrix}1 & 0 \\ 0 & 2 \end{pmatrix}$ & $\begin{pmatrix}2 & 0 \\ 0 & 1\end{pmatrix}$ & $\begin{pmatrix}1 & 0 \\ 1 & 4\end{pmatrix}$ & $\begin{pmatrix}1 & 0 \\ 3 & 4 \end{pmatrix}$ & $ $ & $ $ \\ \hline

27 & $\begin{pmatrix}1 & 0 \\ 0 & 2 \end{pmatrix}$ & $\begin{pmatrix} 2 & 0 \\ 0 & 1\end{pmatrix}$ & $\begin{pmatrix}1 & 0 \\ 1 & 8 \end{pmatrix}$ & $\begin{pmatrix}1 & 0 \\ 3 & 4\end{pmatrix}$ & $\begin{pmatrix}1 & 0 \\ 5 & 8\end{pmatrix}$ & $ $ \\ \hline

28 & $\begin{pmatrix}1 & 0 \\ 0 & 2\end{pmatrix}$ & $\begin{pmatrix} 2 & 0 \\ 0 & 1\end{pmatrix}$ & $\begin{pmatrix}1 & 0 \\ 1 & 16\end{pmatrix}$ & $\begin{pmatrix}1 & 0 \\ 3 & 4 \end{pmatrix}$ & $\begin{pmatrix}1 & 0 \\ 5 & 8\end{pmatrix}$ & $\begin{pmatrix}1 & 0 \\ 9 & 16\end{pmatrix}$ \\ \hline

29 & $\begin{pmatrix}1 & 0 \\ 0 & 2\end{pmatrix}$ & $\begin{pmatrix} 2 & 0 \\ 0 & 1\end{pmatrix}$ & $\begin{pmatrix} 1 & 0 \\ 1 & 8 \end{pmatrix}$ & $\begin{pmatrix}1 & 0 \\ 3 & 4\end{pmatrix}$ & $\begin{pmatrix}1 & 0 \\ 13 & 16 \end{pmatrix}$ & $\begin{pmatrix}1 & 0 \\ 5 & 16\end{pmatrix}$ \\ \hline

30 & $\begin{pmatrix}1 & 0 \\ 0 & 2\end{pmatrix}$ & $\begin{pmatrix} 2 & 0 \\ 0 & 1\end{pmatrix}$ & $\begin{pmatrix}1 & 0 \\ 1 & 4\end{pmatrix}$ & $\begin{pmatrix}1 & 0 \\ 7 & 8\end{pmatrix}$ & $\begin{pmatrix}1 & 0 \\ 3 & 8\end{pmatrix}$ & $ $ \\ \hline

31 & $\begin{pmatrix}1 & 0 \\ 0 & 2\end{pmatrix}$ & $\begin{pmatrix} 2 & 0 \\ 0 & 1\end{pmatrix}$ & $\begin{pmatrix}1 & 0 \\ 1 & 4 \end{pmatrix}$ & $\begin{pmatrix}1 & 0 \\ 15 & 16 \end{pmatrix}$ & $\begin{pmatrix}1 & 0 \\ 3 & 8\end{pmatrix}$ & $\begin{pmatrix}1 & 0 \\ 7 & 16\end{pmatrix}$ \\ \hline

32 & $\begin{pmatrix}1 & 0 \\ 0 & 2 \end{pmatrix}$ & $\begin{pmatrix} 2 & 0 \\ 0 & 1\end{pmatrix}$ & $\begin{pmatrix}1 & 0 \\ 1 & 8\end{pmatrix}$ & $\begin{pmatrix}1 & 0 \\ 7 & 8 \end{pmatrix}$ & $\begin{pmatrix}1 & 0 \\ 3 & 8 \end{pmatrix}$ & $\begin{pmatrix}1 & 0 \\ 5 & 8 \end{pmatrix}$ \\ \hline

33 & $\begin{pmatrix}1 & 0 \\ 0 & 2 \end{pmatrix}$ & $\begin{pmatrix} 2 & 0 \\ 0 & 1\end{pmatrix}$ & $\begin{pmatrix}1 & 0 \\ 1 & 4 \end{pmatrix}$ & $\begin{pmatrix}1 & 0 \\ 7 & 8 \end{pmatrix}$ & $\begin{pmatrix}1 & 0 \\ 3 & 16 \end{pmatrix}$ & $\begin{pmatrix}1 & 0 \\ 11 & 16\end{pmatrix}$ \\ \hline

34 & $\begin{pmatrix}1 & 0 \\ 0 & 2 \end{pmatrix}$ & $\begin{pmatrix} 2 & 0 \\ 0 & 1\end{pmatrix}$ & $\begin{pmatrix}1 & 0 \\ 1 & 6 \end{pmatrix}$ & $\begin{pmatrix}1 & 0 \\ 5 & 6 \end{pmatrix}$ & $\begin{pmatrix}1 & 0 \\ 3 & 6 \end{pmatrix}$ & $\begin{pmatrix} 3 & 0 \\ 1 & 2\end{pmatrix}$ \\ \hline

35 & $\begin{pmatrix}1 & 0 \\ 0 & 2\end{pmatrix}$ & $\begin{pmatrix} 4 & 0 \\ 0 & 1\end{pmatrix}$ & $\begin{pmatrix} 1 & 0 \\ 1 & 4 \end{pmatrix}$ & $\begin{pmatrix} 2 & 0 \\ 1& 2 \end{pmatrix}$ & $\begin{pmatrix}1 & 0 \\ 3 & 4\end{pmatrix}$ & $ $ \\ \hline

36 & $\begin{pmatrix}1 & 0 \\ 0 & 2\end{pmatrix}$ & $\begin{pmatrix} 8 & 0 \\ 0 & 1\end{pmatrix}$ & $\begin{pmatrix}1 & 0 \\ 1 & 4\end{pmatrix}$ & $\begin{pmatrix}1 & 0 \\ 3 & 4\end{pmatrix}$ & $\begin{pmatrix}2 & 0 \\ 1 & 2\end{pmatrix}$ & $\begin{pmatrix} 4 & 0 \\ 1 & 2 \end{pmatrix}$ \\ \hline

37 & $\begin{pmatrix}1 & 0 \\ 0 & 2\end{pmatrix}$ & $\begin{pmatrix} 4 & 0 \\ 0 & 1\end{pmatrix}$ & $\begin{pmatrix}1 & 0 \\ 1 & 4\end{pmatrix}$ & $\begin{pmatrix}1 & 0 \\ 3 & 4\end{pmatrix}$ & $\begin{pmatrix} 2 & 0 \\ 1 & 4\end{pmatrix}$ & $\begin{pmatrix} 2 & 0 \\ 3 & 4\end{pmatrix}$ \\ \hline

38 & $\begin{pmatrix}1 & 0 \\ 0 & 2\end{pmatrix}$ & $\begin{pmatrix} 4 & 0 \\ 0 & 1\end{pmatrix}$ & $\begin{pmatrix}1 & 0 \\ 1 & 8 \end{pmatrix}$ & $\begin{pmatrix}1 & 0 \\ 3 & 4\end{pmatrix}$ & $\begin{pmatrix} 2 & 0 \\ 1 & 2\end{pmatrix}$ & $\begin{pmatrix} 1 & 0 \\ 5 & 8\end{pmatrix}$ \\ \hline

39 & $\begin{pmatrix}1 & 0 \\ 0 & 2\end{pmatrix}$ & $\begin{pmatrix}4 & 0 \\ 0 & 1\end{pmatrix}$ & $\begin{pmatrix}1 & 0 \\ 1 & 4\end{pmatrix}$ & $\begin{pmatrix}1 & 0 \\ 7 & 8\end{pmatrix}$ & $\begin{pmatrix}2 & 0 \\ 1 & 2\end{pmatrix}$ & $\begin{pmatrix}1 & 0 \\ 3 & 8\end{pmatrix}$ \\ \hline

40 & $\begin{pmatrix}1 & 0 \\ 0 & 3 \end{pmatrix}$ & $\begin{pmatrix} 3 & 0 \\ 0 & 1\end{pmatrix}$ & $\begin{pmatrix}1 & 0 \\ 1 & 6\end{pmatrix}$ & $\begin{pmatrix}1 & 0 \\ 2 & 3\end{pmatrix}$ & $\begin{pmatrix}1 & 0 \\ 4 & 6\end{pmatrix}$ & $\begin{pmatrix} 2 & 0 \\ 2 & 3\end{pmatrix}$ \\ \hline

41 & $\begin{pmatrix}1 & 0 \\ 0 & 3\end{pmatrix}$ & $\begin{pmatrix} 3 & 0 \\ 0 & 1\end{pmatrix}$ & $\begin{pmatrix}1 & 0 \\ 1 & 9\end{pmatrix}$ & $\begin{pmatrix}1 & 0 \\ 2 & 3\end{pmatrix}$ & $\begin{pmatrix}1 & 0 \\ 4 & 9\end{pmatrix}$ & $\begin{pmatrix}1 & 0 \\ 7 & 9\end{pmatrix}$ \\ \hline

42 & $\begin{pmatrix}1 & 0 \\ 0 & 3\end{pmatrix}$ & $\begin{pmatrix}6 & 0 \\ 0 & 1\end{pmatrix}$ & $\begin{pmatrix}1 & 0 \\ 1 & 3\end{pmatrix}$ & $\begin{pmatrix}1 & 0 \\ 2 & 3\end{pmatrix}$ & $\begin{pmatrix}3 & 0 \\ 1 & 2\end{pmatrix}$ & $\begin{pmatrix}3 & 0 \\ 0 & 2\end{pmatrix}$ \\ \hline

43 & $\begin{pmatrix}1 & 0 \\ 0 & 3\end{pmatrix}$ & $\begin{pmatrix} 9 & 0 \\ 0 & 1\end{pmatrix}$ & $\begin{pmatrix}1 & 0 \\ 1 & 3\end{pmatrix}$ & $\begin{pmatrix}1 & 0 \\ 2 & 3\end{pmatrix}$ & $\begin{pmatrix} 3 & 0 \\ 1 & 3\end{pmatrix}$ & $\begin{pmatrix} 3 & 0 \\ 2 & 3\end{pmatrix}$ \\ \hline

44 & $\begin{pmatrix}1 & 0 \\ 0 & 6\end{pmatrix}$ & $\begin{pmatrix} 3 & 0 \\ 0 & 1\end{pmatrix}$ & $\begin{pmatrix}1 & 0 \\ 1 & 3\end{pmatrix}$ & $\begin{pmatrix}1 & 0 \\ 2 & 3\end{pmatrix}$ & $\begin{pmatrix}1 & 0 \\ 3 & 6\end{pmatrix}$ & $\begin{pmatrix} 2 & 0 \\ 0 & 3\end{pmatrix}$ \\ \hline

45 & $\begin{pmatrix}1 & 0 \\ 0 & 9\end{pmatrix}$ & $\begin{pmatrix} 3 & 0 \\ 0 & 1\end{pmatrix}$ & $\begin{pmatrix}1 & 0 \\ 1 & 3\end{pmatrix}$ & $\begin{pmatrix}1 & 0 \\ 2 & 3 \end{pmatrix}$ & $\begin{pmatrix}1 & 0 \\ 3 & 9\end{pmatrix}$ & $\begin{pmatrix}1 & 0 \\ 6 & 9\end{pmatrix}$ \\ \hline

46 & $\begin{pmatrix}1 & 0 \\ 0 & 4\end{pmatrix}$ & $\begin{pmatrix}2 & 0 \\ 0 & 1\end{pmatrix}$ & $\begin{pmatrix}1 & 0 \\ 1 & 4\end{pmatrix}$ & $\begin{pmatrix}1 & 0 \\ 2 & 4\end{pmatrix}$ & $\begin{pmatrix}1 & 0 \\ 3 & 4\end{pmatrix}$ & $ $ \\ \hline

47 & $\begin{pmatrix}1 & 0 \\ 0 & 4 \end{pmatrix}$ & $\begin{pmatrix}2 & 0 \\ 0 & 1\end{pmatrix}$ & $\begin{pmatrix}1 & 0 \\ 1 & 4\end{pmatrix}$ & $\begin{pmatrix}1 & 0 \\ 3 & 4\end{pmatrix}$ & $\begin{pmatrix}1 & 0 \\ 2 & 8\end{pmatrix}$ & $\begin{pmatrix}1 & 0 \\ 6 & 8\end{pmatrix}$ \\ \hline

48 & $\begin{pmatrix}1 & 0 \\ 0 & 8\end{pmatrix}$ & $\begin{pmatrix} 2 & 0 \\ 0 & 1\end{pmatrix}$ & $\begin{pmatrix}1 & 0 \\ 1 & 4 \end{pmatrix}$ & $\begin{pmatrix}1 & 0 \\ 3 & 4\end{pmatrix}$ & $\begin{pmatrix}1 & 0 \\ 2 & 4 \end{pmatrix}$ & $\begin{pmatrix}1 & 0 \\ 4 & 8\end{pmatrix}$ \\ \hline

49 & $\begin{pmatrix}1 & 0 \\ 0 & 4\end{pmatrix}$ & $\begin{pmatrix} 2 & 0 \\ 0 & 1\end{pmatrix}$ & $\begin{pmatrix}1 & 0 \\ 1 & 8\end{pmatrix}$ & $\begin{pmatrix}1 & 0 \\ 3 & 4 \end{pmatrix}$ & $\begin{pmatrix}1 & 0 \\ 2 & 4\end{pmatrix}$ & $\begin{pmatrix}1 & 0 \\ 5 & 8\end{pmatrix}$ \\ \hline

50 & $\begin{pmatrix}1 & 0 \\ 0 & 4 \end{pmatrix}$ & $\begin{pmatrix} 2 & 0 \\ 0 & 1\end{pmatrix}$ & $\begin{pmatrix}1 & 0 \\ 1 & 4 \end{pmatrix}$ & $\begin{pmatrix}1 & 0 \\ 7 & 8 \end{pmatrix}$ & $\begin{pmatrix}1 & 0 \\ 2 & 4 \end{pmatrix}$ & $\begin{pmatrix}1 & 0 \\ 3 & 8\end{pmatrix}$ \\ \hline

51 & $\begin{pmatrix}1 & 0 \\ 0 & 3\end{pmatrix}$ & $\begin{pmatrix}3 & 0 \\ 0 & 1\end{pmatrix}$ & $\begin{pmatrix}1 & 0 \\ 1 & 3\end{pmatrix}$ & $\begin{pmatrix}1 & 0 \\ 5 & 6\end{pmatrix}$ & $\begin{pmatrix}1 & 0 \\ 2 & 6\end{pmatrix}$ & $\begin{pmatrix} 2 & 0 \\ 1 & 3\end{pmatrix}$ \\ \hline

52 & $\begin{pmatrix}1 & 0 \\ 0 & 3\end{pmatrix}$ & $\begin{pmatrix}3 & 0 \\ 0 & 1\end{pmatrix}$ & $\begin{pmatrix}1 & 0 \\ 1 & 3\end{pmatrix}$ & $\begin{pmatrix}1 & 0 \\ 8 & 9\end{pmatrix}$ & $\begin{pmatrix}1 & 0 \\ 2 & 9\end{pmatrix}$ & $\begin{pmatrix}1 & 0 \\ 5 & 9\end{pmatrix}$ \\ \hline

53 & $\begin{pmatrix}1 & 0 \\ 0 & 4\end{pmatrix}$ & $\begin{pmatrix}4 & 0 \\ 0 & 1\end{pmatrix}$ & $\begin{pmatrix}1 & 0 \\ 1 & 4\end{pmatrix}$ & $\begin{pmatrix}1 & 0 \\ 3 & 4\end{pmatrix}$ & $\begin{pmatrix}1 & 0 \\ 2 & 4\end{pmatrix}$ & $\begin{pmatrix}2 & 0 \\ 1 & 2\end{pmatrix}$ \\ \hline

54 & $\begin{pmatrix}1 & 0 \\ 0 & 5\end{pmatrix}$ & $\begin{pmatrix}5 & 0 \\ 0 & 1\end{pmatrix}$ & $\begin{pmatrix}1 & 0 \\ 1 & 5\end{pmatrix}$ & $\begin{pmatrix}1 & 0 \\ 4 & 5\end{pmatrix}$ & $\begin{pmatrix}1 & 0 \\ 2 & 5\end{pmatrix}$ & $\begin{pmatrix}1 & 0 \\ 3 & 5\end{pmatrix}$ \\ \hline 

\caption{Minimal coverings, up to permutation, of lengths between $3$ and $6$.}
\label{table1}
\end{longtable}
\end{center}

\end{document}